\documentclass[11pt]{amsart}

\usepackage{bm}
\usepackage{amsfonts, amsmath, amsthm, amssymb, mathrsfs, setspace}%,
\usepackage{geometry}%[bottom=1cm]
\usepackage{url}
\usepackage[all]{xy}
\usepackage{graphicx, epstopdf}%, epstofig}
\usepackage{color}

\numberwithin{equation}{section}
\makeatletter
\newif\if@restonecol
\makeatother

\newtheorem{theorem}{Theorem}[section]
  \newtheorem{lemma}[theorem]{Lemma}
  \newtheorem*{problem}{Problem}
  
   \newtheorem{corollary}[theorem]{Corollary}
  \newtheorem{proposition}[theorem]{Proposition}

  \theoremstyle{definition}
  \newtheorem{definition}[theorem]{Definition}
  \newtheorem{example}[theorem]{Example}
   \newtheorem{remark}[theorem]{Remark}

\usepackage[boxed,vlined,lined]{algorithm2e}

\newcommand \NN{\mathbb{N}} 
 
\newcommand \ZZ {\mathbb{Z}}
\newcommand \CC {\mathbb{C}}

\newcommand \pr {\mathbb{P}} 
\newcommand \A{\mathfrak{a}}

\newcommand \Ann{\textrm{Ann}}

\newcommand \RR {\mathbb{R}}

\newcommand \Sec {\textrm{Sec}}

\newcommand \rk {\textrm{rk}}

\newcommand \Sn {\mathbb{S}}
\newcommand \TropC {\mathcal{T}}
\newcommand \init {\textrm{in}}
\newcommand \NP {\textrm{NP}}

\newcommand \Ass {\textrm{Ass}}
\newcommand \cT {\mathcal{T}}
\newcommand \cM {\mathcal{M}}
\newcommand \cP {\mathcal{P}}
\newcommand \cS {\mathcal{S}}
\newcommand \cQ {\mathcal{Q}}
\newcommand \cF {\mathcal{F}}
\newcommand \x {\underline{x}}
\newcommand \y {\underline{y}}
\newcommand \sat{\textrm{sat}}

\newcommand \gfan {\texttt{gfan}}
\newcommand \polymake {\texttt{Polymake}}

\makeatletter
\@namedef{subjclassname@2010}{\textup{2010} Mathematics Subject Classification}
\makeatother

\begin{document}

 \title[An Implicitization Challenge for Binary Factor Analysis]{An Implicitization Challenge\\ for Binary Factor Analysis}
  \author{Mar\'ia Ang\'elica Cueto}
  \address{Department of Mathematics, University of California,
    Berkeley, CA 94720, USA.} 
\email{macueto@math.berkeley.edu} 
\author{Enrique A.~Tobis}
\address{Departamento de Matem\'atica, FCEN - Universidad de Buenos Aires, 
Pabell\'on I - Ciudad Universitaria, C1428EGA, Buenos Aires, Argentina.}
\email{etobis@dc.uba.ar}
\author{Josephine Yu}
\address{School of Mathematics, Georgia Institute of Technology, Atlanta GA 30332, USA.}
\email{josephine.yu@math.gatech.edu}
\thanks{M.A.\ Cueto was supported by a UC Berkeley Chancellor's
  Fellowship. E.A.\ Tobis was supported by a CONICET Doctoral
  Fellowship, CONICET PIP 5617, ANPCyT PICT 20569 and UBACyT X042 and
  X064 grants. J.\ Yu was supported by an NSF Postdoctoral Fellowship.}

\keywords{ 
    Factor analysis, tropical geometry, Hadamard products, Newton
    polytope
}

\subjclass[2010]{14T05,{(14M25,14Q10)}}

\begin{abstract}
  We use tropical geometry to compute the multidegree and Newton polytope of the
  hypersurface of a statistical model with two hidden and four
  observed binary random variables, solving an open question stated by
  Drton, Sturmfels and Sullivant in \cite[Problem 7.7]{LAS}.  The model
  is obtained from the undirected graphical model of the complete
  bipartite graph $K_{2,4}$ by marginalizing two of the six binary
  random variables. We present algorithms for computing the Newton
  polytope of its defining equation by parallel walks along the
  polytope and its normal fan. In this way we compute vertices of
  the polytope. Finally, we also compute and certify its
  facets by studying tangent cones of the polytope at the symmetry classes vertices.  The Newton polytope has \emph{17\,214\,912} vertices in \emph{44\,938} symmetry classes and \emph{70\,646} facets in \emph{246} symmetry classes.
\end{abstract}

\maketitle

\section{Introduction}\label{sec:intro}

In recent years, a fruitful interaction between (computational)
algebraic geometry and statistics has emerged, under the form of
algebraic statistics. The main objects studied by this field are
probability distributions that can be described by means of polynomial
or even rational maps. Among them, an important source of examples are
the so called graphical models. In this paper, we focus our attention
on a special model: the \emph{undirected $(4,2)$-binary factor analysis model}
$\mathcal{F}_{4,2}$.

First, let us describe our main player. Consider the complete
undirected bipartite graph $K_{2,4}$ with four {\em observed} nodes
$X_1, X_2, X_3, X_4$ and two {\em hidden} nodes $H_1, H_2$
(cf.~Figure \ref{fig:K24}).  Each node represents a binary random
variable and each edge represents a dependency between two random
variables.  In other words, if there is no edge between two random
variables, then they are conditionally independent given the rest of
the variables.  We obtain a hidden model from this undirected
graphical model by marginalizing over $H_1$ and $H_2$. This model is
the discrete undirected version of the factor analysis model
discussed in  \cite[Section~4.2]{LAS}.  The model
and its immediate generalization $\mathcal{F}_{m,n}$ is closely
related to the statistical model describing the behavior of restricted
Boltzmann machines~\cite{RBM}, which are widely discussed in the
Machine Learning literature. Here, $\mathcal{F}_{m,n}$ is the binary
undirected graphical model with $n$ hidden variables and $m$ observed
variables encoded in the complete bipartite graph $K_{m,n}$. The main
invariant of interest in these models is the expected dimension, and,
furthermore, lower bounds on $n$ such that the probability
distributions are a dense subset of the probability simplex
$\Delta_{2^m-1}$. By direct computation, it is easy to show that
$\mathcal{F}_{2,2}$ and $\mathcal{F}_{3,2}$ are dense subsets of the
corresponding probability simplices, so $\mathcal{F}_{4,2}$ is the
first interesting example worth studying.  Understanding the model
$\mathcal{F}_{4,2}$ can pave the way for the study of restricted
Boltzmann machines in general~\cite{TropRBM}.

\begin{figure}[ht]
\centering
\includegraphics[scale=0.5]{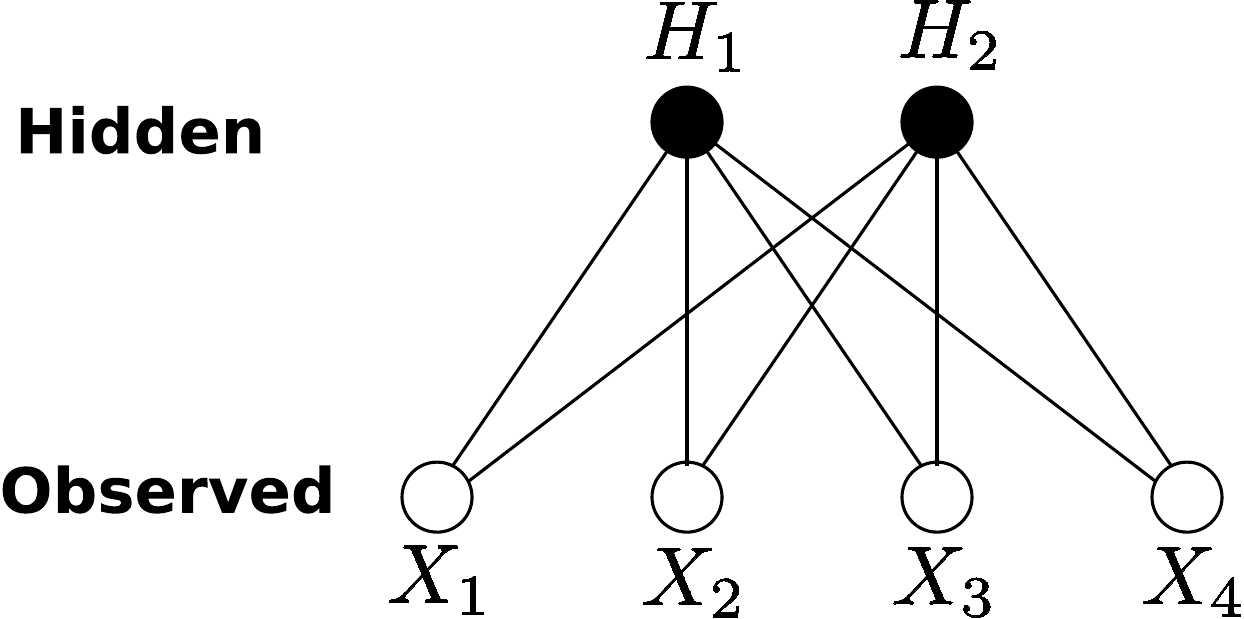}
\caption{The model $\mathcal{F}_{4,2}$.  Each node represents a binary random variable.}
\label{fig:K24}
\end{figure}

The set of all possible joint probability distributions $(X_1, X_2,
X_3, X_4)$ that arise in this way forms a semialgebraic variety $\cM$
in the probability simplex $\Delta_{15}$. To simplify our
construction, we disregard the inequalities defining the model and we
extend our parameterization to the entire affine space $\CC^{16}$.  In
other words, we consider the Zariski closure of the joint probability
distributions in $\CC^{16}$.  As a result of this, we obtain an
algebraic subvariety of $\CC^{16}$ which carries the core information
of our model. In turn, we projectivize the model by considering its
associated projective variety. This variety is expected to have
codimension one and be defined by a homogeneous polynomial in 16
variables.

\begin{problem}(An Implicitization Challenge, \cite[Ch.~VI, Problem
  7.7]{LAS}) Find the degree and the defining polynomial of the model
  $\cM$.
\end{problem}

Our main results  state that the variety $\cM$
is a hypersurface of degree 110 in $\pr^{15}$ (Theorem \ref{thm:main}) and explicitly enumerate all vertices and facets of the polytope (Theorem \ref{thm:polytope}). Our methods are based on tropical geometry.  Since the polynomial is
multihomogeneous, we get its \emph{multidegree} from just one
vertex. Interpolation techniques will allow us to compute the
corresponding irreducible homogeneous polynomial in 16 variables, using the lattice points in the Newton polytope.  However, this polytope will turn out to be too big for interpolation to be practically feasible.

The paper is organized as follows. In Section~\ref{sec:theModel} we
describe the parametric form of our model and we express our variety
as the Hadamard square of the first secant of the Segre embedding 
$\pr^1\times \pr^1\times \pr^1\times \pr^1 \hookrightarrow \pr^{15}$.  In
Section~\ref{sec:tropicalLand} we present the tropical interpretation
of our variety. By means of the nice interplay between the
construction described in Section~\ref{sec:theModel} and its
tropicalization, we compute this tropical variety as a collection of
cones with multiplicities. We should remark that we do not obtain a
fan structure, but, nonetheless, our characterization is sufficient to
fulfill the goal of the paper. The key ingredient is the computation
of multiplicities by the so called push-forward formula~\cite[Theorem
3.12]{NPImplicitEquation} which we generalize to match our setting
(Theorem~\ref{thm:STGeneralized}). We finish
Section~\ref{sec:tropicalLand} by describing the effective computation
of the tropical variety and discussing some of the underlying
combinatorics.

In Section~\ref{sec:ComputeNP} we compute the multidegree of our model
with respect to a natural 5-dimensional grading, which comes from the
tropical picture in Section~\ref{sec:tropicalLand}. Once this question
is answered, we shift gears and move to the study of the Newton
polytope of our variety.  We present two algorithms that compute
vertices of this polytope by ``shooting rays''
(Algorithm~\ref{alg:ray}) and ``walking'' from vertex to vertex in the
Newton polytope (Algorithm~\ref{alg:walk}). Using these methods, and also taking advantage of the $B_4$ symmetry of the polynomial and the Newton polytope, we
compute all \emph{17\,214\,912} vertices our polytope (in \emph{P44\,938} orbits under $B_4$), which shows the
intrinsic difficulties of this ``challenging'' problem. 
Along the way, we also compute the
tangent cones at each symmetry class of vertices and certify the facet normal
directions by looking at the local behavior of the tropical variety
around these vectors (after certifying they belong to the tropical
variety). In particular, by computing dimensions of a certain linear
space (Algorithm~\ref{alg:certifyFacets}) we can check if the vector
is a ray of the tropical variety. In this way, we certify all 246 facets
of the polytope modulo symmetry.
We believe these methods will pave
the way to attack combinatorial questions about high dimensional
polytopes with symmetry as the one analyzed in this paper.

%%%%%%%%%%%%%%%%%%%%%%%%%%%%%%%%%%%%%%%%%%%%%%%%%%%%%%%%%%%%%%%%%%%%%%%%

\section{Geometry of the model}\label{sec:theModel}

%%%%%%%%%%%%%%%%%%%%%%%%%%%%%%%%%%%%%%%%%%%%%%%%%%%%%%%%%%%%%%%%%%%%%%%%

We start this section by describing the parametric representation of
the model we wish to study.  Recall that all our six random variables
are binary, with four observed nodes and two hidden ones. Since the
model comes from an undirected graph (see \cite{LAS, ASCB}), we can
parameterize it by a map $p\colon \RR^{32} \to \RR^{16}$, where
\[
 p_{ijkl} = \sum_{s=0}^1 \sum_{r=0}^1
a_{si}b_{sj}c_{sk}d_{sl} e_{ri}f_{rj}g_{rk}h_{rl} \qquad \textrm{ for
  all }
 (i,j,k,l) \in \{0,1\}^4.
\]
Notice that our coordinates are homogeneous of degree 1 in the subset
of variables corresponding to each edge of the graph. Therefore, there
is a natural interpretation of this model in projective space. On the
other hand, by the distributive law we can write down each coordinate
as a product of two points in the model corresponding to the 4-claw
tree, which is the first secant variety of the Segre embedding $\pr^1
\times \pr^1 \times \pr^1 \times \pr^1 \hookrightarrow \pr^{15}$
(\cite{PhyAlgGeom}), i.e. %.  Namely,
\[
p\colon (\pr^1\times \pr^1)^8 \to \pr^{15} \quad p_{ijkl} = (\sum_{s=0}^1
a_{si}b_{sj}c_{sk}d_{sl})\, ( \sum_{r=0}^1
e_{ri}f_{rj}g_{rk}h_{rl}) \;\; \forall\,
(i,j,k,l) \in
\{0,1\}^4.
\]
From this observation it is natural to consider the Hadamard product
of projective varieties:
\begin{definition}\label{def:star}
  Let $X,Y\subset \pr^{n-1}$ be two projective varieties.  The
  \emph{Hadamard product} of $X$ and $Y$ is
\[
X \centerdot Y = \overline{\{(x_0y_0: \ldots:
    x_{n-1}y_{n-1}) \, |\, x\in X, y\in Y, x \centerdot y\neq 0\}}\subset \pr^{n-1},
\]
where $x \centerdot y = (x_0y_0, \ldots, x_{n-1}y_{n-1})\in \CC^{n}$.
\end{definition}
Note that this structure is well-defined since each coordinate is
bihomogeneous of degree (1,1).  The next proposition follows from the
construction.

\begin{proposition}\label{pr:HadamardProduct}
  The algebraic variety of the model is $\cM = X \centerdot X$ where
  $X$ is the first secant variety of the Segre embedding $\pr^1 \times
  \pr^1 \times \pr^1 \times \pr^1 \hookrightarrow \pr^{15}$.
\end{proposition}

Notice that the binary nature of our random variables enables us to
define a natural $\Sn_2$-action by permuting the values $0$ and $1$ on
each index in our 4-tuples. Combining this with the $\Sn_4$-action on
the 4-tuples of indices, we see that our model comes equipped with a
natural $\Sn_4 \ltimes (\Sn_2)^4$-action.  In other words, the 16
coordinates $p_{ijkl}$ of $\pr^{15}$, for $i,j,k,l \in \{0,1\}$, are
in natural bijection with the vertices of a 4-dimensional
cube. Assuming $\cM$ is a hypersurface (as we will prove in
Section~\ref{sec:tropicalLand}), its
defining polynomial is invariant under the group $B_4$ of
symmetries of the 4-cube, which has order 384.  This group action will be \emph{extremely}
helpful for our computations in the next two sections.

We now describe the ideal associated to the secant variety
$\Sec(\pr^1\times \pr^1\times \pr^1 \times \pr^1)$.  The Segre
embedding $\pr^1\times \pr^1\times \pr^1\times \pr^1 \hookrightarrow
\pr^{15}$ has a monomial parameterization $p_{ijkl}=u_i \cdot v_j
\cdot w_k \cdot x_l$ for $i,j,k,l\in \{0,1\}$.  Its defining prime
ideal is generated by the $2\times 2$-minors of all three $4\times
4$-flattenings, together with some $2\times 2$-minors of the $2\times
8$-flattenings \cite[Section 3]{PhyAlgGeom}:
\[F_{(12|34)}\!:=\!\!
\left(\begin{array}[l]{cccc}
  \!\!p_{0000}\! & \!p_{0001}\! & \!p_{0010}\! &\! p_{0011} \!\! \\
  \!\!p_{0100}\! & \!p_{0101}\! & \!p_{0110}\! & \!p_{0111} \!\! \\ 
 \!\!p_{1000}\! & \!p_{1001}\! & \!p_{1010} \!& \!p_{0111} \!\! \\
  \!\!p_{1100}\! & \!p_{1101}\! &\! p_{1110} \!& \!p_{1111}\!\! 
\end{array}
\right),~~
F_{(13|24)}\!:=\!\!
\left(\begin{array}[c]{cccc}
\!\!   p_{0000} \!&\! p_{0001}\! & \!p_{0100}\! & \! p_{0101} \!\! \\
\!\!   p_{0010}\! & \!p_{0011}\! & \!p_{0110}\! & \! p_{0111} \!\! \\
\!\!   p_{1000}\! & \!p_{1001}\! & \!p_{1100}\! & \! p_{1101} \!\!  \\
\!\!   p_{1010}\! &\! p_{1011}\! &\! p_{1110}\! &\! p_{1111} \!\!  
\end{array}
\right)
,\]
\[
F_{(14|23)}\!:=\!\!
\left(
\begin{array}[r]{cccc}
\!\!     p_{0000}\! & \!p_{0010}\! & \!p_{0100}\! & \! p_{0110} \!\! \\
\!\!  p_{0001}\! &\! p_{0011}\! & \! p_{0101}\! & \!p_{0111} \!\! \\
\!\!   p_{1000}\! & \!p_{1010}\! & \!p_{1100}\! & \!p_{1110} \!\! \\
\!\!   p_{1001}\! & \!p_{1011}\! & \!p_{1101}\! & \!p_{1111} \!\! 
\end{array}
\right).
\]
In turn, the defining ideal of the first secant variety of the Segre
embedding can be computed from the previous three $4\times
4$-flattening matrices. We state the result for the case of the
variety we are studying, although the set-theoretic result is also
true for an arbitrary number of observed nodes.
\begin{theorem} [\cite{LM,LW}]
  The secant variety $X=\Sec(\pr^1\times \pr^1\times \pr^1\times
  \pr^1)\subset \pr^{15}$ is the nine-dimensional \emph{irreducible}
  subvariety consisting of all $2\times 2\times 2 \times 2$-tensors of
  tensor rank at most 2.  The prime ideal of $X$ is generated by all
  the $3\times 3$-minors of the three flattenings. 
\end{theorem}

%%%%%%%%%%%%%%%%%%%%%%%%%%%%%%%%%%%%%%%%%%%%%%%%%%%%%%%%%%%%%%%%%%%%%%%%

\section{Tropicalizing the model}\label{sec:tropicalLand}

%%%%%%%%%%%%%%%%%%%%%%%%%%%%%%%%%%%%%%%%%%%%%%%%%%%%%%%%%%%%%%%%%%%%%%%%

In this section we define tropicalizations of varieties in $\CC^n$ and
compute the tropicalization of $\cM$.  See
\cite{BJSST, FirstSteps} for more details about tropical varieties.

\begin{definition}
  For an algebraic variety $X\subset \CC^n$ not contained in a
  coordinate hyperplane and  with defining ideal
  $I=I(X)\subset K[x_1, \ldots, x_n]$, the \emph{tropicalization} of
  $X$ or $I$ is defined as:
\[
\TropC(X) = \cT(I) = \{ w\in \RR^{n} \, |\, \init_w(I) \textrm{ contains no
  monomial}\},
\]
where $\init_w(I)=\langle \init_w(f): f\in I \rangle$, and
$\init_w(f)$ is the sum of all nonzero terms of $f=\sum_\alpha
c_{\alpha} x^{\alpha}$ such that $\alpha \cdot w$ is maximum.

Alternatively, when working with subvarieties of tori $V \subset
(\CC^*)^n$ we consider the defining ideal $I$ over the ring of Laurent
polynomials and set
\[
\cT(V)=\cT(I)=\{ w\in \RR^{n} \, |\, \init_w(I)\neq \langle 1\rangle \}.
\]
Both definitions agree if we consider $X$ to be the Zariski closure of $V$ in
$\CC^n$. We would go back and forth between these two definitions.
\end{definition}

The tropical variety $\cT(I)$ is a polyhedral subfan of the
Gr\"{o}bner fan of $I$.  If $I$ is a prime ideal containing no monomials, then $\cT(I)$ is pure of the
same dimension as $X$ and is connected in codimension one
\cite{BJSST}.
The set $\{w\in \TropC(I) : \init_w (I) = I \}$ is a linear space in
$\RR^n$ and is called the \emph{lineality space} of the fan
$\cT(I)$ or the {\em homogeneity space} of the ideal $I$. This space
can be spanned by integer vectors, which form a primitive lattice
$\Lambda$. This lattice encodes the action of a maximal torus
on $X$, given by a diagonal action. All cones in
$\cT(I)$ contain this linear space.

In addition to their polyhedral structure, tropical varieties are
equipped with integer positive weights on all of their maximal
cones. We now explain how these numbers can be constructed.  A point
$w \in \cT(I)$ is called {\em regular} if $\cT(I)$ is a linear space
locally near $w$.  The {\em multiplicity} $m_{w}$ of a regular point
$w$ is the sum of multiplicities of all minimal associated primes of
the initial ideal $\init_w(I)$.  See \cite[Section 3.6]{CommAlg} for
definitions.  The multiplicity of a maximal cone $\sigma \subset
\cT(I)$ is defined to be equal to $m_w$ for any $w \in \sigma$ in its
relative interior.  It can be showed that this assignment does not
depend on the choice of $w$.  With these multiplicities, the tropical
variety satisfies the {\em balancing condition} \cite{ElimTheory}.

As we discussed in the previous section
(Proposition~\ref{pr:HadamardProduct}) our variety is expressed as a
Hadamard power of a well-known variety. This Hadamard square has a
dense set which can be parameterized in terms of a monomial map (the
coordinatewise product of two points).  The integer matrix of
exponents corresponding to this monomial map is $(I_{n}\mid
I_n)\subset \ZZ^{n\times 2n}$. Although tropicalization is not
functorial in general, it has nice properties if we restrict it to
monomial maps between subvarieties of tori.

We now describe the tropicalization of monomial maps.  Let $A$ be a $d
\times r$ integer matrix defining a monomial map $\alpha\colon (\CC^*)^r
\rightarrow (\CC^*)^d$ and a linear map $A\colon \RR^r \rightarrow \RR^d$
defined by left multiplication by this matrix.
\begin{theorem}\cite{ElimTheory, TrIm}
\label{thm:ST}
Let $V \subset (\CC^*)^r$ be a subvariety.  Then 
$$
\cT(\alpha(V)) ~~ = ~~ A(\cT(V)).
$$
Moreover, if $\alpha$ induces a generically finite morphism of degree
$\delta$ on $V$, then the multiplicity of $\cT(\alpha(V))$ at a
regular point $w$ is
$$
m_w = \frac{1}{\delta} \cdot \sum_v m_v \cdot \text{ index
}(\mathbb{L}_w \cap \ZZ^d : A(\mathbb{L}_v \cap \ZZ^r)),
$$
where the sum is over all points $v \in \cT(V)$ with $A v = w$.  We
also assume that the number of such $v$ is finite, all of them are
regular in $\cT(V)$, and $\mathbb{L}_v, \mathbb{L}_w$ are linear spans
of neighborhoods of $v \in \cT(V)$ and $w \in A \cT(V)$ respectively.
\end{theorem}

At first sight, the hypothesis of this theorem is not satisfied by
our variety because the map $\alpha_{|_{X\times X}}$ is not generically
finite. However it is very close to having this finiteness
behavior. Namely, after taking the quotient $X'$ of $X$ by a maximal
torus action, and a choice of a suitable monomial map
$\overline{\alpha}$, the map $\overline{\alpha}_{|_{X'\times X'}}$ becomes
generically finite and we can apply Theorem~\ref{thm:ST}. We now
explain this reduction process.

Let $V\subset (\CC^*)^r$ a subvariety, $\alpha\colon (\CC^*)^r \to
(\CC^*)^d$ a monomial map, and let $W=\alpha(V)$. Consider the
lineality space $\RR\otimes_{\ZZ}\Lambda\subset \cT(V)$, and let
$\Lambda'=A(\Lambda)$. We identify $\RR\otimes_{\ZZ}\Lambda$ with a
$\ZZ$-basis of $\Lambda= (\RR \otimes_{\ZZ}\Lambda)\cap \ZZ^r$. Notice
that $\Lambda'$ need not be a primitive lattice in $\ZZ^d$ in
general. Call $(\Lambda')^{\sat}$ its saturation in $\ZZ^d$, that is
$(\Lambda')^{\sat}=(\RR\otimes_{\ZZ} \Lambda') \cap \ZZ^d$. 
We know by construction and Theorem~\ref{thm:ST} that
$\RR\otimes_{\ZZ}\Lambda'$ is contained in the lineality space of
$\cT(W)$. Therefore, we can consider the linear map between these
tropical varieties after moding out by $\RR\otimes_{\ZZ}\Lambda$ and
$\RR\otimes_{\ZZ}\Lambda'$ respectively. As we mentioned earlier, the
lineality space of each tropical variety determines the maximal torus
action. For example, $(\CC^*)^r$ acts on $V$ by $t\cdot (x_1, \ldots,
x_n) :=(t^{a_1}x_1, \ldots, t^{a_r}x_r)$ where $\underline{a}$ lies in
$\Lambda$.

The linear map $A$ sends $\Lambda$ onto $\Lambda'$, inside
the lineality space of $\cT(\alpha(V))$. In addition, the monomial map
$\alpha$ is compatible with the torus actions on $V$ and $\alpha(V)$.
In particular, the equality $\alpha(\Lambda\otimes_{\ZZ}\CC^*)=
\Lambda'\otimes_{\ZZ} \CC^*$ induces an action on $W$ by a subtorus
(the one corresponding to the primitive lattice
$(\Lambda')^{\sat}$). Thus, we can take the quotient of $V$ and $W$ by
the corresponding actions of tori $H$ and $H'$. We obtain the
commutative diagram:
\begin{equation}
%\setstretch{0.65}
{\xymatrix{
V\ar@{->>}[d]_{\pi} \ar@{->>}[r]^{\alpha} & W\ar@{->>}[d]^{\pi}\\
V'=V/H \ar@{->>}[r]^- {\bar{\alpha}} & W/H'=W'.
}}%\vspace{1ex}
\label{eq:1}
\end{equation}
Here, $H=\Lambda \otimes_{\ZZ} \CC^*\cong (\CC^*)^{\dim \Lambda}$ and
$H'=\Lambda'\otimes_{\ZZ} \CC^*\cong (\CC^*)^{\dim
  \Lambda'}$.  Since $\Lambda$ is a primitive
sublattice of $\ZZ^r$, it admits a primitive complement in
$\ZZ^r$. Fix one of them and call it $\Lambda^{\perp}$. Note that
this complement need not be the usual orthogonal complement.

Assume for simplicity that $\Lambda'$ is a \emph{primitive} sublattice of
$\ZZ^d$.  Therefore, we can identify $\bar{\alpha}$ with the monomial
map corresponding to the linear map:
\[
A'\colon (\RR\otimes \ZZ^r)/(\RR\otimes \Lambda) =
\RR\otimes \Lambda^{\perp}=:\!(\RR\otimes \Lambda)^{\perp}\! \to
(\RR\otimes \ZZ^d/(\RR\otimes \Lambda')
=\! \RR\otimes \Lambda'^{\perp}\!\!=:(\RR\otimes \Lambda')^{\perp}.
\]
 Since $\Lambda$ is primitive,  $(\RR\otimes \Lambda)^{\perp} \cap \ZZ^r=\Lambda^\perp$, and likewise
 for $\Lambda'^{\perp}$.

 To simplify notation, call $L:=\RR\otimes \Lambda$ and
 $L':=\RR\otimes \Lambda'$. From the construction it is easy to see
 that $\cT(V') = \cT(V)/L$ and $\cT(W')=\cT(W)/L'$ as sets. But in
 fact, they agree as weighted balanced polyhedral fans. More
 precisely,

\begin{lemma}\label{lm:MultiModOutByLineality}
  Let $X\subset (\CC^*)^{r}$ and let $L$ be a subspace of the lineality
  space of the tropical variety $\cT(X)$ generated by \emph{integer}
  vectors. Then $\cT(X)/L$ is a balanced weighted polyhedral fan where
  the multiplicities at regular points ${w}'$ are defined as
  $m_{{w}'}=m_{w}$ for any $w$ in the fiber of ${w}'$ under the
  projection map.  With these weights, $\cT(X)/L$ coincides with the
  tropical variety $\cT(X')$, where $X'$ is the quotient of $X$ by the
  torus $(L\cap \ZZ^{r})\otimes_{\ZZ} \CC^*\cong (\CC^*)^{\dim L}$,
  which is a subtorus of the maximal torus acting on $X$.
\end{lemma}
\begin{proof}
  By definition, we know that $\init_{w+L}(I)=\init_w(I)$ for any
  $w\in \RR^r$. Let $l:=\dim L$. Call  $\Lambda:=L\cap \ZZ^r$ the
  underlying lattice of $L$. Since $\Lambda$ is a primitive
  lattice, we can extend any $\ZZ$-basis of $\Lambda$ to a $\ZZ$-basis of
  $\ZZ^r$.  Thus, after a linear change of coordinates (i.e.~a
  monomial change of coordinates given by this new $\ZZ$-basis of
  $\ZZ^r$) we can assume $\Lambda=\ZZ\langle e_1, \ldots, e_l\rangle$. And in
  this case, we can pick the direct summand $\Lambda^{\perp}$ of
  $\Lambda$ to be $\ZZ\langle e_{l+1}, \ldots, e_{r}\rangle$.  In
  particular, the projection map $\pi\colon X\to X'=X/H$ corresponds
  to the monomial map $\alpha\colon (\CC^*)^r\to (\CC^*)^{r-l}$
  determined by the integer matrix $A \in \ZZ^{r\times (r-l)}$, whose
  columns are a $\ZZ$-basis of $\Lambda^{\perp}$.
  
  By construction, $I=I(X)\subset \CC[x_1^{\pm 1}, \ldots, x_r^{\pm
    1}]$ is homogeneous with respect to the grading $\deg(x_i)=e_i$ for $i\leq
  l$ and $\deg(x_j)=\underline{0}$ for $j>l$. Since any homogeneous Laurent
  polynomial is of the form $f=\underline{x}^{\alpha} g(x_{l+1},
  \ldots, x_r)$, we see that $I$ is generated by Laurent polynomials
  in the variables $\{x_{l+1}, \ldots x_r\}$. Call $g_1, \ldots, g_s$
  these generators.  Therefore $I'=I(X')=\langle g_1(x_{l+1}, \ldots,
  x_r), \ldots g_s(x_{l+1}, \ldots, x_r)\rangle \subset
  \CC[x_{l+1}^{\pm 1}, \ldots, x_r^{\pm 1}]$ and $I=I'\CC[x_1^{\pm 1},
  \ldots, x_r^{\pm 1}]$.

  From Theorem~\ref{thm:ST} we know that $\cT (X')= A\cT (X)=\cT(X)/L$
  as sets. Moreover, since the subspace $L$ lies in all cones of $\cT
  (X)$, then the set $\cT (X')$ which is the quotient of $\cT (X)$ by
  $L$ has a natural fan structure inherited from the one of $\cT(
  X)$. By definition, if ${w}'$ is a regular point in $\cT(X')$ then
  any lifting point in $w+L$ would be a regular point in
  $\cT(X)$. Moreover, $\init_w(I)=\init_{{w}'}(I')\CC[x_1^{\pm 1},
  \ldots, x_r^{\pm 1}]$. In particular, a primary decomposition
  $\init_{{w}'}(I')$ determines a primary decomposition of
  $\init_w(I)$ by extending each ideal to the whole Laurent polynomial
  ring in $n$ variables. Therefore, to show $m_{{w}'}=m_w$ it suffices
  to show that the multiplicity of any minimal prime $P\subset
  \CC[x_{l+1}^{\pm 1}, \ldots, x_r^{\pm 1}]$ of $\init_{{w}'}(I')$
  equals the multiplicity of $P\subset \CC[x_1^{\pm 1}, \ldots,
  x_r^{\pm 1}]$ in $\init_{w}(I)$. This claim follows from the
  definition of multiplicity. More precisely:
\begin{align*}
  m(P, \init_{{w}'}(I'))& = \dim_{\frac{S_P}{PS_P}}
  \frac{S_P}{S_P\init_{{w}'(I')}}= \dim_{(\frac{S_P}{PS_P})[x_{1}^{\pm
      1}\!\!\!,
    \ldots, x_l^{\pm 1}]} \frac{S_P[x_{1}^{\pm 1}\!\!\!, \ldots, x_l^{\pm
      1}]}{S_P[x_{1}^{\pm 1}\!\!\!, \ldots, x_l^{\pm
      1}]\init_{{w}'}(I')}\\ 
& =  \dim_{\frac{S[x_{1}^{\pm 1}\!\!\!, \ldots, x_l^{\pm
        1}]_P}{PS[x_{1}^{\pm 1}\!\!\!, \ldots, x_l^{\pm 1}]_P}}
  \frac{S[x_{1}^{\pm 1}\!\!\!, \ldots, x_l^{\pm 1}]_P}{S[x_{1}^{\pm 1}\!\!\!,
    \ldots, x_l^{\pm 1}]_P\init_w(I)}=m(P,\init_w(I)),
\end{align*}
where $S=\CC[x_{l+1}^{\pm 1}, \ldots,
x_r^{\pm 1}]$.
\end{proof}
Using the previous construction, we extend Theorem~\ref{thm:ST} to the
case of monomial maps that are generically finite after taking
quotients by appropriate tori. This extension fits perfectly into our
setting.

\begin{theorem} \label{thm:STGeneralized} Let $\alpha\colon  (\CC^*)^r \to
  (\CC^*)^d$ be a monomial map with associated integer matrix $A$ and let
  $V\subset (\CC^*)^r$ be a closed subvariety.
Then,
\[
\cT(\alpha(V)) ~~ = ~~ A(\cT(V)).
\]
Suppose $V$ has a torus action given by a rank $l$ lattice
$\Lambda\subset \ZZ^r$.  Let $V'$ be the quotient by this torus
action. Let $\overline{\alpha}\colon V'\to
(\CC^*)^d/\alpha(\Lambda\otimes_{\ZZ} \CC^*)$ be the induced monomial
map, with associated
integer matrix $A'$.\\
Suppose $\Lambda'=A(\Lambda)$ is a primitive sublattice of $\ZZ^d$ and
that $\bar{\alpha}$ induces a generically finite morphism of degree
$\delta$ on $V'$. Then the multiplicity of $\cT(\alpha(V))$ at a
regular point $w$ can be computed as:
\begin{equation}
m_w = \frac{1}{\delta} \cdot \sum_{\substack{\pi(v)\\A\cdot v=w}} m_v \cdot \text{ index
}(\mathbb{L}_{w} \cap \ZZ^d :
A(\mathbb{L}_{v} \cap \ZZ^r)),\label{eq:3}
\end{equation}
where the sum is over any set of representatives of points
$\{v'=\pi(v)\in \cT(V') \mid A'v'=w'\}$ given $w'=\pi(w)\in
\RR^d/(\RR\otimes_{\ZZ}\Lambda')=\RR\otimes_{\ZZ}\Lambda'^{\perp}$. We
also assume that the number of such $v'$ is finite, all of them are
regular in $\cT(V')$ and $\mathbb{L}_{v}, \mathbb{L}_{w}$ are linear
spans of neighborhoods of $v \in \cT(V)$ and $w\in A\cT(V)$
respectively.
\end{theorem}

\begin{remark}
  In case $\Lambda'$ is not a primitive lattice, the formula for
  $m_w$ will involve an extra factor, namely, the index of
  $\Lambda'$ with respect to its saturation $\Lambda'^{\,\sat}$ in
  $\ZZ^d$. In this case, $\Lambda'^{\perp}$ will correspond to any
  complement of the primitive lattice $\Lambda'^{\,\sat}$ inside
  $\ZZ^d$.
\end{remark}

\begin{proof}[\textbf{Proof of 
Theorem~\ref{thm:STGeneralized}.}]
The equality as sets follows from Theorem~\ref{thm:ST}.  To prove the
formula for multiplicities, we first note that the sum in~\eqref{eq:3}
is finite. This follows because $\bar{\alpha}$ induces a generically
finite morphism if and only if $\ker{A'} \cap
\cT(V')=\{\underline{0}\}$ if and only if $A(\Lambda^{\perp}) \cap
\Lambda'=\{\underline{0}\}$.
  
From the diagram \eqref{eq:1} and the surjectivity of $\alpha$ and
$\bar{\alpha}$, we know that the multiplicity formula holds for
$\cT(Y')$ and the morphism $\bar{\alpha}$.  Pick $w'$ a regular point
of $\cT(X')$ and pick any point $w$ in the fiber
$\pi^{-1}(w')=w+(\RR\otimes \Lambda')$. By definition, $w$ is a
regular point of $\cT(X)$ and we have $m_w =m_{w'}$ by
Lemma~\ref{lm:MultiModOutByLineality}. We assume all $v'$ in the
fiber of $A'$ at $w'$ are regular in $\cT(V')$ and
$\mathbb{L}_{\pi(v)}, \mathbb{L}_{\pi(w)}$ are linear spans of
neighborhoods of $\pi(v) \in \cT(V')$ and $\pi(w)\in A'\cT(V')$
respectively.

  By construction, the index set in the formula for $m_{w'}$ agrees
  with the index set in formula~\eqref{eq:3} for $m_w$.
  Therefore, our goal would be to show that each summand indexed by
  $\pi(v)$ in the formula for $m_{w'}$ equals its corresponding summand in
  formula~\eqref{eq:3} for $m_w$. We know that $m_v=m_{\pi(v)}$ by
  Lemma~\ref{lm:MultiModOutByLineality}.
  Therefore, we only need to prove that the lattice indices on each
  summand are the same, i.e.

  \begin{equation}
    \label{eq:5}
    \text{ index
    }(\mathbb{L}_w \cap \ZZ^d : A(\mathbb{L}_v \cap \ZZ^r)) =
    \text{ index
    }(\mathbb{L}_{\pi(w)} \cap (\Lambda'^{\perp}) :
    A'(\mathbb{L}_{\pi(v)} \cap \Lambda^{\perp})).
  \end{equation}

Note that by construction, $\Lambda'\subset \mathbb{L}_w\cap \, \ZZ^d$,
$\Lambda\subset \mathbb{L}_v$, and likewise $A(\Lambda)=\Lambda'\subset
A(\mathbb{L}_v \cap \ZZ^r)$.
 Hence, we can consider the quotient of $\mathbb{L}_w\cap \ZZ^d$ and
 $A(\mathbb{L}_v \cap \ZZ^r)$ by $\Lambda'$. We obtain
 \[
 \frac{\mathbb{L}_w\cap \ZZ^d}
{ A(\mathbb{L}_v \cap \ZZ^r)} \cong
 \frac{(\mathbb{L}_w\cap \ZZ^d)/\Lambda'}{
 A(\mathbb{L}_v \cap \ZZ^r)/\Lambda'}.
\]
The equality in~\eqref{eq:5} follows by the identifications
$(\mathbb{L}_w\cap \ZZ^d)/\Lambda'=\mathbb{L}_{\pi(w)} \cap
(\Lambda'^{\perp})$ and ${ A(\mathbb{L}_v \cap
  \ZZ^n)/\Lambda'}=A'(\mathbb{L}_{\pi(v)}\cap \Lambda^{\perp})$,
via projecting to $\Lambda'^{\perp}$.
\end{proof}

\begin{theorem}\label{pr:cartesianProdAndMult}
  Given $X,Y\subset \CC^N$ two irreducible varieties, consider the
  associated variety $X \times Y\subset \CC^{2N}$. Then
\[
\cT(X\times Y) = \cT(X)\times \cT(Y)
\]
as weighted polyhedral complexes, with $m_{\sigma\times
  \tau}=m_{\sigma}m_{\tau}$ for maximal cones $\sigma \subset \cT(X),
\tau \subset \cT(Y),$ and $\sigma\times \tau \subset \cT(X \times Y)$.
\end{theorem}
\begin{proof}
  The equality as polyhedral complexes is a direct consequence of the
  equality $\init_{(u,v)}(I+J)=\init_{u}I + \init_v J$, which follows
  by Buchberger's criterion and the fact that the generators of $I$
  and $J$ involve disjoint sets of variables.  If we pick $u\in \cT
  X$, $v\in \cT Y$ regular points, then $(u,v)$ is a regular point in
  $\cT (X\times Y)$. Our goal is to prove the multiplicity formula.

  Given two primary decompositions $\init_u(I)=\bigcap_i
  M_i\subset\CC[\x]$, $\init_v(J)=\bigcap_j N_j\subset\CC[\y]$, we
  claim that $\init_{(u,v)}(I+J)=\bigcap_{i,j}
  (M_i+N_j)\subset \CC[\x,\y]$ is also a primary decomposition. 
The equality as sets follows immediately, so we only need to show that
$M_i+N_j\subset \CC[\x,\y]$ is a primary ideal. Let $P_i\subset
\CC[\x]$ and $Q_j\subset \CC[\y]$ be associate prime ideals to $M_i$
and $N_j$ respectively. Since $\CC$ is algebraically closed, and $M_i$
and $N_j$ involved disjoint sets of variables, it is immediate to
check that $P_i + Q_j\subset \CC[\x,\y]$ is a prime ideal. Namely,
the quotient ring $\CC[\x,\y]/(P_i+Q_j)$ equals
$(\CC[\x]/P_i)[\y]\otimes_{\CC} (\CC[\y]/Q_j)[\x]$, a tensor product of
two domains over $\CC$, hence also a domain.

Moreover, since both $M_i$ and $N_j$ involve disjoint sets of
variables, we have
\[
\Ann(M_i+N_j)=\Ann\,M_i\otimes_{\CC} \CC[\y] + \CC[\x]\otimes_{\CC}\Ann\,N_j.
\]
From this and the fact that $P_i^{s_i}\subset\Ann\, M_i\subset P_i$ and
$Q_j^{t_j}\subset \Ann\, N_j\subset Q_j$ for suitable $s_i, t_j\in \NN$,
we conclude $(P_i+Q_j)^{s_i+t_j}\subset \Ann (M_i+N_j)\subset P_i+Q_j$
thus proving by definition that $M_i+N_j$ is a $(P_i+Q_j)$-primary
ideal.

With similar arguments we conclude that all minimal
  primes of $\init_{(u,v)}(I+J)$ are sums of minimal primes of
  $\init_u(I)$ and $\init_v(J)$.  This follows because, given
  $P,P'\subset \CC[\x]$ and $Q,Q'\subset \CC[\y]$ prime ideals, it is
  straightforward to check that $P+Q\subset P'+Q'$ if and only if
  $P\subset P'$ and $Q\subset Q'$.

  Let $\sigma, \tau$ be maximal cones on $\cT(X)$ and $\cT(Y)$, and
  let $u, v$ be regular points in $\sigma$ and $\tau$
  respectively.  By definition of multiplicity of a maximal cone, we
  have
\[
m_{\sigma}=\!\!\!\!\!\!\!\!\!\sum_{\substack{P\in \Ass(\init_u(I))\\ P \textrm{
      minimal}}}\!\!\!\!\!\!\!\!\! m(P, \CC[\underline{x}]/\init_u I)
=
\!\!\!\!\!\!\!\!\!\sum_{\substack{P\in \Ass(\init_u(I))\\ P \textrm{ minimal}}}
\!\!\!\!\!\!\!\!\!\dim_{(\CC[\x]/P)_P}(\CC[\underline{x}]/\init_u
I)_P\;;
\]
\[
m_{\tau} =\!\!\!\!\!\!\!\!\!\!\!\!\! \sum_{\substack{Q\in
    \Ass(\init_v(J))\\ Q \textrm{ minimal}}} \!\!\!\!\!\!\!\!\!\!\!
\dim_{({\CC[\y]}/{Q})_Q} (\CC[\underline{y}]/\init_v J)_Q\; ;\;
m_{\sigma\times \tau} =\!\!\!\!\!\!\!\!\!\!\!\!\! \sum_{\substack{P\in \Ass(\init_u(I))\\
    Q\in \Ass(\init_v(J))\\P, Q \textrm{ minimal}}}\!\!
\!\!\!\!\!\!\!\!\!
\dim_{(\frac{\CC[\x,\y]}{P+Q})_{P+Q}}\!\!\big(\frac{\CC[\x,\y]}{\init_u
  I + \init_v J}\big)_{P+Q}.
\]
The statement $m_{\sigma\times \tau}=m_{\sigma}m_{\tau}$ follows from
the distributive law and Lemma~\ref{lm:prodLenght}.
\end{proof}
\begin{lemma}\label{lm:prodLenght}
Let $I\subset \CC[\x
]$, $J\subset \CC[\y
]$ be ideals
and let $P\subset \CC[\x
]$, $Q\subset \CC[\y
]$ be minimal
primes containing $I$ and $J$ respectively. Then
\[
\dim_{(\CC[\x,\y]/{P+Q})_{P+Q}}\left(\frac{\CC[\x,\y]}{I + J}\right)_{P+Q}=
\dim_{(\CC[\x]/{P})_{P}}(\CC[\x]/I)_P \cdot\dim_{(\CC[\y]/{Q})_{Q}}(\CC[\y]/J)_Q.
\]
\end{lemma}
\begin{proof}
  Consider the residue fields $F=(\CC[\x]/P)_P$, $G=(\CC[\y]/Q)_Q$,
  and $L= (\CC[\x,\y]/(P+Q))_{P+Q}$. Note that $F\otimes_\CC G
  \hookrightarrow L$ via the natural inclusion given by the
  multiplication map, since $\CC$ is algebraically closed.
  Likewise, one can easily show that $\CC[\x]/I\otimes_{\CC}
  \CC[\y]/J\cong \CC[\x,\y]/(I+J)$ via the multiplication map. We
  wish to find a similar result for the localization of these
  quotients at the corresponding minimal primes. 

  For simplicity, call $M=(\CC[\x]/I)_P\cong F^s$ and
  $N=(\CC[\y]/J)_Q\cong G^r$ the corresponding finite dimensional
  vector spaces. Our goal is to prove that $M\otimes_{\CC} N$ is a
  free $L$-vector space of rank $sr$. From the canonical isomorphisms
  $\CC[\x]_P\otimes_{\CC[\x]} \CC[\x]/I\cong (\CC[\x]/I)_P$,
  $\CC[\y]_Q\otimes_{\CC[\y]} \CC[\y]/J\cong (\CC[\x]/J)_Q$, we see
that  $M\otimes_\CC N=(\CC[\x]/I)_P\otimes_\CC (\CC[\y]/J)_Q \cong
  \CC[\x,\y]/(I+J)[S^{-1}]$, where $S=(\CC[\x]\smallsetminus
  P)(\CC[\y]\smallsetminus Q)$ is the multiplicatively closed set
  consisting of products of polynomials, each of which is pure in each
  set of variables, and which do not lie inside the
  prime ideals $P$ or $Q$.
Similarly, $F\otimes_{\CC} G \cong \CC[\x,\y]/(P+Q)[S^{-1}]$.  

On the other hand, notice that $M\otimes_\CC N$ comes with a natural
$F\otimes_\CC G$-module structure via ``coordinatewise action.''
Hence,
\[
(\CC[\x,\y]/(I+J))_{(P+Q)} \cong 
L\otimes_{(F\otimes_\CC G)} (M\otimes_\CC N).
\]
From the last isomorphism we see that to prove our lemma it suffices
to show that $M\otimes_{\CC} N$ is a free $F\otimes_\CC G$-module of
rank $sr$. The original claim will follow after tensoring with $L$.

Let $\{f_i\}$, $\{g_j\}$ be bases of $M$ and $N$ respectively. We claim that
$\{f_i\otimes g_j\}$ is a basis of $M\otimes_{\CC} N$ as an
$F\otimes_{\CC} G$-module. It suffices to check the linear
independence. We proceed in an elementary way, by successively using
the linear independence of the different bases of the free modules $M,
N, F$ and $G$. Suppose $\sum_{i,j} a_{ij} f_i\otimes g_j =0 \in
M\otimes_{\CC} N$, with $a_{ij}\in F\otimes_{\CC}G$. Write
$a_{ij}=\sum_{k,l} a_{ijkl} u_k\otimes v_l$ where $a_{ijkl}\in \CC$
and $u_k, v_l$ are basis elements of the field extensions $F|\CC$, $G|
\CC$ respectively. Thus,
\begin{equation}
0 = \sum_{i,j} a_{ij} f_i\otimes
g_j = \sum_{j,l}\big(\sum_{i,k} a_{ijkl} u_kf_i\big)\otimes_{\CC} (v_l g_j).\label{eq:4}
\end{equation}
To prove $a_{ij}=0$ it suffices to show $a_{ijkl}=0$ for all
$i,j,k,l$.  By a well-know result on tensor algebras (cf.~\cite[Lemma
6.4]{CommAlg}), expression~\eqref{eq:4} implies the existence of
elements $a_{jlt} \in \CC$, $h_t\in M$ such that $\sum_t a_{jlt} h_t =
\sum_{i,k}a_{ijkl}u_kf_i$ for all $j,l$ and $\sum_{j,l} a_{jlt}
v_lg_j=0$ for all $t$.  Hence, rearranging the sum we conclude that
$\sum_{j}(\sum_{l} a_{jlt}v_l)g_j=0$ in $N$ for all $t$, which
implies $\sum_{l} a_{jlt}v_l=0 \in G$ for all $j,t$. This in turn
implies $ a_{jlt}=0$ for all $j,l,t$.

Using the condition $\sum_{i}(\sum_k a_{ijkl}u_k)f_i= \sum_t a_{jlt}
h_t=0$, we have $\sum_{k} a_{ijkl} u_k = 0$ for all $i,j,l$.
Therefore, $a_{ijkl}=0$ for all $i,j,k,l$, as we wanted to show.
\end{proof}

\begin{corollary}\label{cor:key}
  Given $X,Y\subset \pr^n$ two projective irreducible varieties none
  of which is contained in a proper coordinate hyperplane, we can
  consider the associated irreducible projective variety $X \centerdot Y\subset
  \pr^n$. Then as \emph{sets}:
\[
\TropC(X \centerdot Y) = \TropC(X) + \TropC(Y),
\]
where the sum on the right-hand side denotes the Minkowski sum in
$\RR^{n+1}$.
\end{corollary}
As one can easily imagine, this set-theoretic result is motivated by
(and is a direct consequence of) Kapranov's theorem~\cite[Theorem
2.2.5]{EKL} (i.e., the fundamental theorem of tropical geometry) and
the fact that valuations turn products into sums (see \cite[Theorem
2.3]{ElimTheory} for the precise statement). The novelty of our
approach is that under suitable finiteness condition of the monomial
map defining Hadamard products, we can effectively compute
multiplicities of regular points in $\cT(X\centerdot Y)$ from
multiplicities of $\TropC(X)$ and $\TropC(Y)$.  
It is important to mention that this finiteness condition holds for
the example we are studying in this paper.
Moreover, we are \textbf{not} claiming that $\cT (X\centerdot Y)$
inherits a fan structure from $\cT(X)$ and $\cT(Y)$. In general, it might happen that maximal
cones in the Minkowski sum get subdivided to give maximal cones in
$\TropC(X\centerdot Y)$ or, moreover, the union of several cones in
the Minkowski sum gives a maximal cone in $\TropC(X\centerdot Y)$.

\begin{example}
It may seem surprising at first that the combinatorial structure (e.g.\ $f$-vector) of the Newton polytope does not follow easily from the description of the tropical hypersurface as a Minkowski sum of two fans.  Moreover, the number of edges of the polytope (and even the number of vertices) may exceed the number of maximal cones of the tropical hypersurface {\em given as a set}.
To see this in a small example, consider the tropical curve in $\RR^3$ whose six rays are columns of the following matrix

\[
\left(
\begin{array}{cccccc}
1 & 1  & 1& 1  & 1 &-5  \\
0 &  0 & 1& 1 &  2&  -4\\
0 &  1 &  0& 2 &  1 & -4
\end{array}
\right),
\]
and consider the Minkowski sum of the fan with itself.  This tropical
hypersurface is described as a union of 15 cones (or as a non-planar
graph in $\Sn^2$ with 6 nodes and 15 edges), but the dual Newton
polytope has 16 vertices, 25 edges, and 11 facets.  If we intersect
the tropical hypersurface with a sphere around the origin, we would
see the planar graph in
Figure~\ref{fig:TropicalSurfaceWithNOFanStructure}.

\begin{figure}[htb]
  \centering
  \includegraphics[scale=0.5]{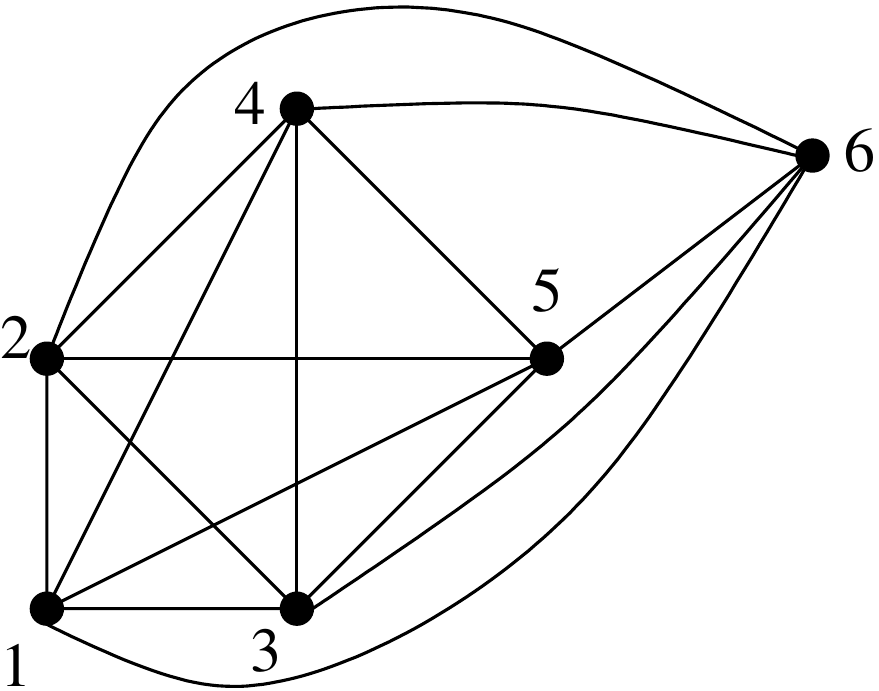}
  \caption{A tropical surface in $\RR^3$ described as a collection of
    2-dimensional cones in $\RR^3$ or as a non-planar graph in $\Sn^2$.}
\label{fig:TropicalSurfaceWithNOFanStructure}
\end{figure}

The planer regions correspond to the 16 vertices.  The black dots
correspond to the columns in the above matrix and the arcs between
them correspond to cones generated by them. The nodes in the graph
correspond to the facets in the Newton polytope. Six of these facets
correspond to the black dots in
Figure~\ref{fig:TropicalSurfaceWithNOFanStructure} and are the 6 nodes
in the non-planar graph description of the tropical hypersurface.  The
remaining five facets correspond to the missing intersection points
between the edges of the non-planar graph in the picture. Adding these
5 nodes to the graph will give us a planar graph with 11 nodes and 25
edges that encodes the fan structure of the tropical variety and the
combinatorics of the Newton polytope.  \smallskip

If we had started instead with a tropical curve whose six rays are $\pm e_i$ for $i=1,2,3$, then the dual polytope would be a cube with $f$-vector $(8, 12, 6)$.
\qed
\end{example}

Due to the lack of a fan structure in our description of $X \centerdot
Y$, Corollary~\ref{cor:key} gives no estimate for the number of
maximal cones in the tropical variety $X\centerdot Y$, where the fan
structure is inherited from the Gr\"obner fan structure of the
defining ideal of $X\centerdot Y$. Moreover, this fan structure is
infeasible to obtain in general. Hence, in the hypersurface case we
have no estimate on the number of edges of the dual polytope to the
tropical variety $\cT (X\centerdot Y)$ and, as a consequence, no
estimate on the number of vertices of the polytope. As the previous
example illustrates, the description of $\cT(X\centerdot Y)$ as a
collection of weighted cones of maximal dimension contains less
combinatorial information than the fan structure does and hence, the
computation of the dual polytope becomes more challenging, as we
show in Section~\ref{sec:ComputeNP}.

\medskip

We now describe the computation of the tropical variety $\cT(\cM)$ of
our model $\cM$.  By our discussions in Section~\ref{sec:theModel}, we
know that the defining ideal of $X = \Sec(\pr^1\times \pr^1\times
\pr^1 \times \pr^1)\subset \pr^{15}$ is generated by the $3 \times 3$
minors of the three flattenings of $2 \times 2 \times 2 \times 2$
matrix of variables $(p_{ijkl})$, for a total of 48 generators. Since
$X$ is irreducible, we can use \gfan\ \cite{gfan} to compute the
tropical variety $\cT(X)$.

The ideal $I(X)$ of $\CC[p_{0000}, \ldots, p_{1111}]$ is invariant
under the action of $B_4$, and \texttt{gfan} can exploit the symmetry
of a variety determined by an action of a subgroup of the symmetric
group $\Sn_{16}$. For this, we need to provide a set of generators as
part of the input data.  The output groups cones together according to
their orbits.

The tropical variety $\cT(X) \in \RR^{16}$ has a lineality space
spanned by the rows of the following integer matrix:
\begin{equation}
\label{eqn:lineality}
\setstretch{0.65}
\Lambda = 
\left(
  \begin{array}{cccccccccccccccc}    1  & 1  & 1  & 1  & 1  & 1  &1   & 1  & 1  & 1  & 1  & 1  & 1  & 1  & 1  & 1  \\   0  & 0  & 0  &0   & 0  &  0 &  0 & 0  &  1 &  1 & 1  & 1  &  1 &  1 & 1  & 1  \\   0  & 0  & 0  & 0  & 1  & 1  &  1 & 1  & 0  & 0  &  0 & 0  &  1 & 1  & 1  & 1  \\    0  & 0  &  1 & 1  & 0  & 0  & 1  & 1  & 0  & 0  & 1  & 1  & 0  & 0  & 1  & 1  \\    0  & 1  &  0 &   1&  0 &   1& 0  &1   &  0 &  1 &   0&  1 &  0 &   1&  0 &  1  \end{array}
\right),
\end{equation}
where the columns correspond to variables $p_{ijkl}$, for $i,j,k,l \in
\{0,1\}$, ordered lexicographically. As we explained already in this
section, we can identify this linear space with the maximal torus
acting on the variety $X$ and hence on $X\centerdot X$. A set of
generators of the corresponding lattice giving this action can be read-off from
the parameterization. More precisely, consider the morphism of tori
$\beta\colon (\CC^*)^5 \to (\CC^*)^{16}$ sending $(t_0, \ldots, t_4)\mapsto
(t^{m_1}, \ldots, t^{m_{16}})$, where each $m_i$ is of the form
$(1,v_i)$, where $v_i$ runs over all sixteen vertices of the
4-cube. Then, one can check that the closure of the image of $\beta$
in $\CC^{16}$ is the affine cone over the Segre embedding $\pr^1\times
\pr^1\times \pr^1\times \pr^1 \hookrightarrow \pr^{15}$. More precisely, given
a generic point in the image of $\beta$, we have $t_0t_1^it_2^jt_3^kt_4^l =
\lambda x_iy_jz_kw_l$, where $(x_0:x_1)=(1:t_1), (y_0:y_1)=(1:t_2),
(z_0:z_1)=(1:t_3), (w_0:w_1)=(1:t_4)\in \pr^1$ and $\lambda=t_0 \in
\RR$.

The {\tt gfan} computation confirms that the tropical variety $\cT(X)$
inside $\RR^{16}$ is a $10$-dimensional polyhedral fan with a
$5$-dimensional lineality space.  After moding out by the lineality
space, the $f$-vector is: $$(382, 3436, 11236, 15640, 7680).$$
Regarding the orbit structure, there are 13 rays and 49 maximal cones
in $\cT(X)$ up to symmetry and all maximal cones have multiplicity 1.

According to Corollary~\ref{cor:key}, the tropical variety of the model
is (as a set)
$$\cT(\cM) = \cT(X \centerdot X) = \cT(X) + \cT(X).$$
Since we know that this will result in a pure polyhedral fan, we only
need to compute all Minkowski sums between pairs of cones of
\emph{maximal} dimension.  For this step we use the $B_4$ group
action.  There is a natural (coordinatewise) action of $B_4 \times
B_4$ on $\TropC(X) \times \TropC(X)$ that translates to a $B_4$-action
on $\TropC(X) + \TropC (X)$. Therefore, to compute the Minkowski sum
of maximal cones, we first consider $49 \cdot $7\,680$ = $ 376\,320
pairs $(\sigma_1, \sigma_2)$, where $\sigma_1$ is taken from a set of
representatives of the $49$ orbits of maximal cones, and $\sigma_2$ is
taken from the set of all maximal cones.  We discard the pairs
$(\sigma_1, \sigma_2)$ for which $\sigma_1 + \sigma_2$ is not of
maximal dimension 15.  After this reduction, the total number of
maximal cones computed is 92\,469.  By construction, this list of
92\,469 cones contains all representatives of the orbits of maximal
cones in $\TropC (X \centerdot X)$.  But they do not form distinct
orbits.  Some cones appear twice in the list as $\sigma + \tau$ and
$\tau + \sigma$, and this the only possibility except for 4\,512 cones
which arise from two different pairs, plus their flips. That is,
$\sigma_1 +\tau_1 = \sigma_2 + \tau_2$ where both pairs differ only by
an interchange of a \emph{single} pair of extremal rays $(r_1,r_2) \in
(\sigma_1, \tau_1)$: i.e.  $\sigma_2=(\sigma_1\smallsetminus \{r_1\})
\cup \{r_2\}$ and $\tau_2=(\tau_1\smallsetminus \{r_2\}) \cup
\{r_1\}$.  Some cones $\sigma$ have non-trivial stabilizers in $B_4$,
so there are cones $\sigma + \tau_1$ and $\sigma + \tau_2$ in the same
orbit. The dimension of the maximal cones in $\cT(\cM)$ confirms that
$\cM$ is a hypersurface.

The total number of orbits of maximal cones is 18\,972, and each orbit
has size 96, 192, or 384.  We then let the group $B_4$ act on each
orbit and obtain 6\,865\,824 cones of dimension 15, the union of which
is the tropical variety $\cT(\cM)$, as predicted by
Corollary~\ref{cor:key}.  We do not have a fan structure of
$\cT(\cM)$.  Nonetheless, we can compute the multiplicity of
$\cT(\cM)$ at any regular point using Theorem \ref{thm:STGeneralized}
because our matrix $A$ is of the form $(I_{16} \mid I_{16})\in
\ZZ^{16\times 32}$.  After taking quotients by the respective maximal
torus acting on each space, the map $X' \times X' \rightarrow X'
\centerdot X'$, is generically finite of degree two.  In practice, the
lattice indices in~\eqref{eq:3} are computed via greatest common
divisors (gcd) of maximal minors of integer matrices whose rows span
the cones in $\cT(X)$ and $\cT(X\times X)$.  More precisely,
\begin{lemma}\label{lm:indexCalculation}
  Given a lattice $D\subset \ZZ^r$, and an integer matrix $A\subset
  \ZZ^{d\times r}$ with $\rk(A(D))=\rk(D)$, the lattice index $\text{index}( \RR\otimes_{\ZZ}A
  (D) \, \cap \ZZ^d : A(D))$ can be computed as follows. Pick $\{w_1,
  \ldots w_s\}$ a minimal system of generators of $D$ over $\ZZ$, and let $B:=(w_1\mid
  \ldots \mid w_s)\in \ZZ^{r\times s}$. Then, the index equals the
  quotient of  the gcd of the maximal minors of the matrix $A\cdot B\in
  \ZZ^{r\times s}$ by the gcd of the
  maximal minors of the matrix $B$.
\end{lemma}
\begin{proof}
  Since $\RR\otimes_{\ZZ}A(D) = \RR \otimes_{\ZZ}
  A(\RR\otimes_{\ZZ} D\, \cap \ZZ^r)$, the
$
\text{index}( \RR\otimes_{\ZZ}A (D) \, \cap \ZZ^d : A(D))$ equals the product\[
\text{index}(
\RR\otimes_{\ZZ} A (D) \cap \ZZ^d : A (\RR\otimes_{\ZZ} D \cap \ZZ^r))
\cdot \text{index} (A(\RR\otimes_{\ZZ} D \cap \ZZ^r) : A(D)).
\]
By construction $\text{index}( \RR\otimes_{\ZZ}A (D) \, \cap \ZZ^d :
A(D))$ is  the gcd of the maximal minors
of the matrix $A\cdot B$. To prove the result, it suffices to show that 
$\text{index} (A(\RR\otimes_{\ZZ} D \cap \ZZ^r) : A(D))$ equals the
gcd of the maximal minors of the matrix $B$ in the statement.

Since $\rk(A(D))=\rk(D)$, this implies that $\ker A\cap
D=\ker A\cap (\RR\otimes_{\ZZ} D\cap \ZZ^r)=\{\underline{0}\}$. Then:
 $A(\RR\otimes_{\ZZ} D \cap \ZZ^r)/A(D) \cong( \RR\otimes_{\ZZ}
D \cap \ZZ^r) /D$, which equals the gcd of the maximal minors of the
matrix $B$, as we wanted to show.
\end{proof}

In our case, $B$ is  spanned by twenty integer vectors (five from each cone
$\sigma\times {\bf 0} , {\bf 0 }\times \tau \in \cT X\times \cT X$ plus the lattices $\Lambda\times {\bf 0}, {\bf
  0} \times \Lambda$ coming from the
lineality space. Call $C_{\sigma}$ and $C_{\tau}$ each list of five
vectors of $\sigma$ and $\tau$. Then, the matrix $B$ in the previous lemma equals the block
diagonal matrix $B=$diag$(B_\sigma, B_\tau)$, where
$B_\sigma=(C_{\sigma}|\Lambda)$, $B_\tau=(C_{\tau}|\Lambda)$ and
$A\cdot B = (C_{\sigma} | \Lambda|C_{\tau} |\Lambda)$.  Thus, the index
equal the quotient of $\gcd(15\times 15$-minors of
$(C_{\sigma}|C_{\tau}|\Lambda) $ by the product
$\gcd \big(10\times 10$-minors of $(C_{\sigma}|\Lambda)\big) \cdot \gcd
\big(10\times 10$-minors of $(C_{\tau}|\Lambda) \big)$.  Each gcd calculation
is done via the Hermite (alt.\ Smith) normal form of these matrices
\cite{Maple}.  After computing all multiplicities we obtain only
values one or two.

\section{Newton polytope of the defining equation}\label{sec:ComputeNP}

In this section, we focus our attention on the \emph{inverse
  problem}. That is, given the tropical fan of an irreducible
hypersurface, we wish to computing the \emph{Newton polytope} of the
defining equation $f=\sum_a c_{a}\underline{x}^a$ of the hypersurface, i.e.\ the convex
hull of all vectors $a \in \ZZ^{16}$ such that $\underline{x}^a$ appears
with a nonzero coefficient in $f$.  

\subsection{Vertices and Facets}

We will first present the results of our computation before discussing
algorithms and implementation in the following subsections.  Here is the ultimate result:

\begin{theorem}
\label{thm:polytope}
The Newton polytope of the defining equation of $\cM$ has 17\,214\,912 
vertices in 44\,938 orbits and 70\,646 facets in 246 orbits under the symmetry group $B_4$.
\end{theorem}

Among the 44\,938 orbits of vertices, 215 have size 192 and 44\,723
has size 384.  The maximum coordinate of a vertex ranges between 14
and 20, and the minimum coordinate is either 0 or 1.  All but 46
orbits have a zero-coordinate.  A vertex can have up to seven
zero-coordinates.  Each vertex is contained in 11 to 62 facets. There
are 11\,800 symmetry classes of {\em simple} vertices, that is, those
contained in exactly 11 facets.  The following is a representative of
the the unique symmetry
class of vertices contained in 62 facets each, which has size 192:
$$(0,0,1,17,13,6,17,1,17,1,6,13,1,17,0,0).$$
% This vertex has orbit size 192.  

We index the coordinates of $\mathbb{P}^{15}$ by $\{0,1\}^4$ and order
them lexicographically. Since our polynomial is (multi)-homogeneous,
knowing even a single point in the Newton polytope gives the
multidegree. We now describe the multidegree of the hypersurface $\cM$:

\begin{theorem}
\label{thm:main}
  The hypersurface $\cM$ has multidegree $(110, 55, 55, 55, 55)$ with
  respect to the grading defined by the matrix in \eqref{eqn:lineality}.
\end{theorem}

Now let us look at the 246 orbits of facets.
The following table lists the orbit sizes:

\smallskip
\begin{center}
  \begin{tabular}{cccccccccccc}
    \hline    \hline
    size & 2 & 8 &12 & 16 & 24 & 32 &48 & 64 & 96 & 192 & 384\\
    number of facet orbits & 1& 2 &1 &3 &1 &1 &7 &3 &15 &67&145\\
    \hline \hline
  \end{tabular}
\end{center}
\smallskip

The coordinates $x_{ijkl}$ are naturally indexed by bit strings $ijkl \in \{0,1\}^4$.  The two facet inequalities in the size-2 orbit 
say that the sum of $ x_{ijkl}$ such that $ i+j+k+l $ is even (or odd) is at least 32.
Each facet contains between 210 and 3\,907\,356 vertices.  The unique symmetry class of facets containing the most vertices consist of coordinate hyperplanes.

Using Algorithm~\ref{alg:certifyFacets}, we certified that out
of the 13 orbits of rays of the 9-dimensional tropical variety of the
Segre embedding $\pr^1\times\pr^1\times \pr^1\times \pr^1\hookrightarrow
\pr^{15}$, only the following eight are facet directions of
$\cT(\cM)$:

\smallskip
\small
\begin{verbatim}
(1, 0, 0, 1, 0, 1, 1, 2, 2, 1, 1, 0, 1, 0, 0, 1)
(1, 3, 3, 1, 3, 1, 1, 3, 1, 3, 3, 1, 3, 1, 1, 3)
(2, 1, 1, 0, 1, 0, 0, 0, 2, 1, 1, 0, 1, 0, 0, 0)
(2, 1, 1, 2, 1, 2, 2, 1, 1, 2, 2, 1, 2, 1, 1, 2)
(3, 2, 2, 1, 2, 1, 1, 0, 2, 1, 1, 0, 1, 0, 0, 0)
(3, 3, 3, 3, 3, 3, 3, 3, 1, 3, 3, 1, 3, 1, 1, 3)
(-1, 0, 0, 0, 0, 0, 0, 0, 0, 0, 0, 0, 0, 0, 0, 0)
(-1, -1, -1, -1, 0, 0, 0, 0, 0, 0, 0, 0, -1, -1, -1, -1).
\end{verbatim}
\normalsize
\smallskip

A complete list of vertices and facets, together with the scripts used
for computation, are available at \smallskip
\begin{center}
  \url{http://people.math.gatech.edu/~jyu67/ImpChallenge/}
\end{center}

\subsection{Computing vertices}

We now discuss how we obtained the Newton polytope.  We will first explain the connection between $\cT(f)$ and $\NP(f)$. From the
tropicalization $\cT(\cM)$ of the hypersurface $\cM = \{p : f(p) = 0
\} \subset \pr^{15}$ we want to compute the extreme monomials of $f$.
For a vector $w \in \RR^{16}$, the initial form $\init_w(f)$ is a
monomial if and only if $w$ is in the interior of a maximal cone
(chamber) of the normal fan of $\NP(f)$.  The tropical variety of the
hypersurface $\cM$ is the union of codimension one cones of the normal
fan of $\NP(f)$.  The multiplicity of a maximal cone in $\cT(\cM)$ is
the lattice length of the edge of $\NP(f)$ normal to that cone.

A construction for the vertices of the Newton polytope $\NP(f)$ from its normal fan
$\TropC (f)$ equipped with multiplicities was developed in
\cite{TropDiscr} (see also \cite{Binomials} for several numerical
examples).  The following is a special case of \cite[Theorem
2.2]{TropDiscr}. Since the operation $ \cT(f)$ interprets $f$ as a
Laurent polynomial, $\NP(f)$ will be determined from $\cT(f)$ up to
translation. The algorithm described in Theorem~\ref{thm:RayShoot}
computes a representative of $\NP(f)$ which lies in the positive
orthant and touches all coordinate hyperplanes, i.e.\ $f$ is a
polynomial not divisible by any non-constant monomial. We describe the
pseudocode in Algorithm~\ref{alg:ray}.

\begin{theorem} \label{thm:RayShoot}
  Suppose $w \in \RR^n$ is a generic vector so that the ray $(w -
  \RR_{>0}\, e_i)$ intersects $\cT(f)$ only at regular points of
  $\cT(f)$, for all $i$.  Let $\cP^w$ be the vertex of the polytope
  $\cP = \NP(f)$ that attains the maximum of $\{w \cdot x : x \in \cP\}$.
  Then the $i^\text{th}$ coordinate of $\cP^w$ equals
$$
\sum_{v} m_v \cdot |l^v_i|,
$$
where the sum is taken over all points $v \in \cT(f) \cap (w - \RR_{>
  0} e_i)$, $m_v$ is the multiplicity of $v$ in $\cT(f)$, and
$l^v_i$ is the $i^\text{th}$ coordinate of the primitive integral
normal vector $l^v$ to the maximal cone in $\cT(f)$ containing $v$.
\end{theorem}

Note that we do not need a fan structure on $\cT(f)$ to use
Theorem~\ref{thm:RayShoot}.  A description of $\cT(f)$ as a set,
together with a way to compute the multiplicities at regular points,
gives us enough information to compute vertices of $\NP(f)$ in any
generic direction.

In Section~\ref{sec:tropicalLand} we computed $\cT(f)$ as a union of
6\,865\,824 cones.  For each of those cones, we calculated the lattice
index in Theorem~\ref{thm:STGeneralized} and the primitive vector
which is the direction of the edge of $\NP(f)$ normal to the cone.
There are 15\,788 distinct edge directions in $\NP(f)$.  We then pick
a random vector $w \in \RR^{16}$ and go through the list of
6\,865\,824 cones, recording the cones that meet any of the rays $w -
\RR_{>0}\, e_i$.  For each $i$, we sum the numbers $m_v \cdot
|l^{v}_i|$ over all the intersection points $v$ and obtain the
$i^\text{th}$ coordinate of the vertex.

\begin{algorithm}[htb]

  \KwIn{The list $\cF$ of maximal cones, with multiplicities,
    whose union is the codimension one cones in the normal fan of a polytope $\cP \subset \RR^n$.  An objective vector $w \in \RR^n$.}
  {\bf Assumption:}  The objective vector $w$ does not lie in any cone in $\cF$, i.e.\ the face $P^w$ is a vertex.  For each $i = 1, 2, \dots, n$ the ray $w - \RR_{>0} e_i$ does not meet the boundary of any cone in $\cF$.  \\
  \KwOut{The vertex $\cP^w$ that maximizes the scalar product with the objective vector $w$.}  
  \smallskip  
  $P^w \leftarrow 0$\\
  \For{each cone $\sigma$ in $\cF$}{
	\For{i = 1, 2, \dots, n}{
		\If{$\sigma \cap (w - \RR_{>0} e_i) \neq \emptyset$}
		{$\cP^w_i \leftarrow \cP^w_i + m_\sigma \cdot \ell_{\sigma,i}$, where $m_\sigma$ is the multiplicity of $\sigma$ and $\ell^\sigma$ is the primitive integral normal vector to $\sigma$ such that $\ell^{\sigma}_i > 0$.}
	}
  }  
  \Return $P^w$.
  \caption{Ray-Shooting: computing a vertex of a polytope from its normal fan.
    \label{alg:ray}}
\end{algorithm}

\begin{figure}[htb]
  \centering
  \includegraphics[scale=0.6]{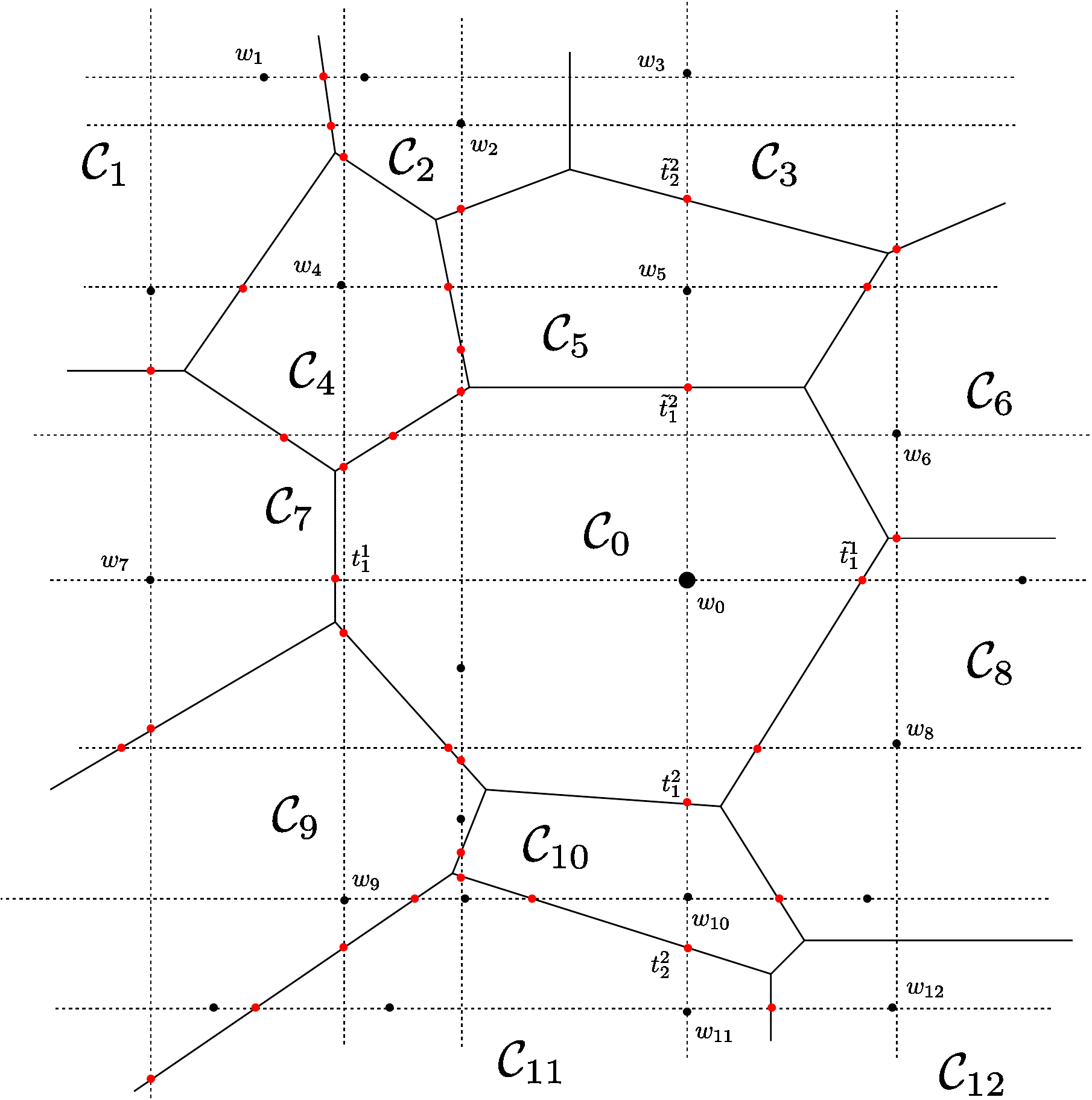}
  \caption{Ray-shooting and walking algorithms combined.  Starting from chamber $\mathcal{C}_0$ we shoot and walk from chamber to chamber.}
  \label{fig:walkAndShoot}
\end{figure}

To obtain the multidegree, we only need one vertex.  We computed the first vertex using \texttt{Macaulay 2} \cite{M2} in a few days.
Our ultimate goal was to compute the Newton polytope $\NP(f)$, a much more difficult computational problem that took us many more months to complete. 
As a first attempt, we bound the number of lattice
points in the polytope by the number of nonnegative lattice points of
the given multidegree. 
Using the software {\tt LattE}
\cite{LattE}, we found that the number of monomials in 16 variables
with multidegree $(110, 55, 55, 55, 55)$ is
\emph{5\,529\,528\,561\,944}. 

By construction, it is clear that the bottleneck of Algorithm~\ref{alg:ray} is in going through the list $\cF$ of 6\,865\,824 cones.
We can modify the algorithm to produce more than one vertex for each
pass through the list.  We do this in two ways.  One is to process
multiple objective vectors at once and save time by reducing the
number of file readings and reusing the linear algebra computations
for checking whether a cone meets a ray or not.  Another way to produce more
vertices is to keep track of the cones that we meet while
ray-shooting, and use them to walk from chamber to chamber in the
normal fan of $\NP(f)$.  This is described in Algorithm~\ref{alg:walk}. 
\begin{algorithm}
\KwIn{A generic objective vector $w \in \RR^n$, the vertex $\cP^w$, and the set $\cS := \{ (\sigma,i,t) \in \cF \times \{1,2,\dots,n\} \times \RR_{>0} ~:~  \sigma \cap (w - \RR_{>0} e_i) = \{w - t e_i\} \}$.  (This input is typically obtained from Algorithm \ref{alg:ray}.)}
  \KwOut{The set of all vertices of $\cP$ with objective vectors of the form $ w - t e_i $ for some $t \in \RR_{>0}$ and $i \in \{1,2, \dots, n\}$.}  
  \smallskip  
	\For{$i = 1, 2, \dots, n$}{
		Let  $\sigma_1, \dots, \sigma_m$ be the cones that intersect the ray $w - \RR_{>0} e_i$ transversely.\\
	Let $t_1, \dots, t_m \in \RR_{>0}$ be such that $(\sigma_k, i, t_k) \in \cS$ for $k = 1,2,\dots,m$. \\
		Order $\sigma_1, \dots, \sigma_m$ so that
 $$t_{1} = \cdots = t_{k_1} < t_{k_1+1} = \cdots = t_{k_2} < \cdots < t_{k_{l}+1} = \cdots = t_{m}:=t_{k_{l+1}} .$$
 	$v \leftarrow P^w$\\
 	\For{$j = 1,2, \dots, l+1$}{
		$\ell^{\sigma_{k_j}}\leftarrow $ primitive integral
                normal vector to $\sigma_{k_j}$ with $\ell^{\sigma_{k_j}}_i>0$;
\\
		$\displaystyle  v \leftarrow v - \left(\sum_{k_{j-1} < k \leq k_j} m_{\sigma_k}\right) \cdot \ell^{\sigma_{k_j}}$, where $k_0 = 1$, $k_{l+1} = m$, and $m_\sigma$ denotes the multiplicity of $\sigma$.\\
		{\bf Output} $v$, and an objective vector in the line segment between $w - t_{k_j} e_i$ and $w - t_{k_{j+1}} e_i$, where $t_{k_{l+2}} := \infty$.
	}
			}
  \caption{Walking: starting from an objective vector and corresponding vertex, compute the vertices obtained by changing the objective vector in negative coordinate directions.
    \label{alg:walk}}
\end{algorithm}
\begin{figure}[htb]
  \centering
  \includegraphics[scale=0.35]{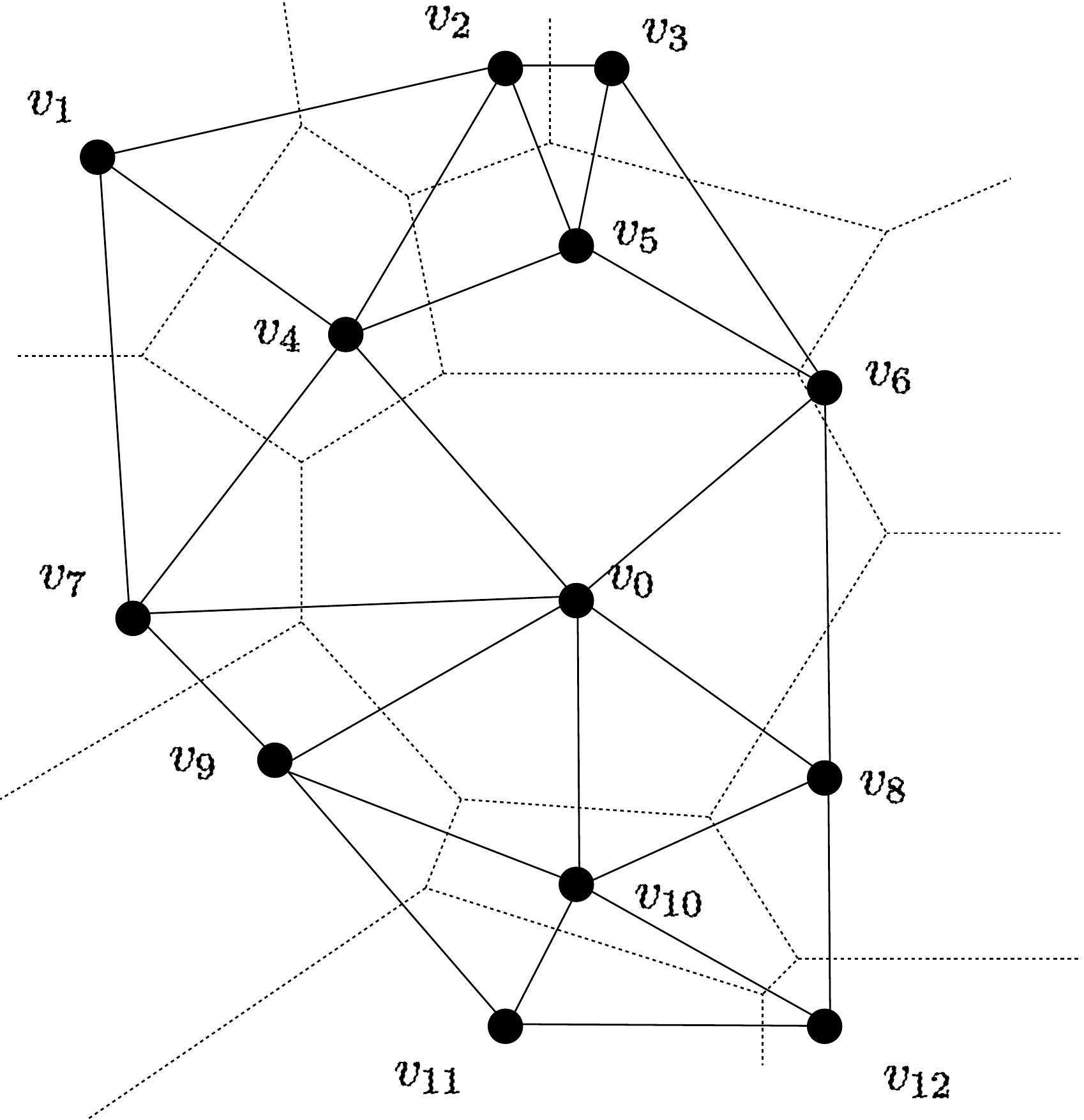}
  \caption{Walking from vertex to vertex in $\NP(f)\subset \RR^3$. In
    dash lines, we plot the tropical variety. The picture represents
    the local structure around $v_0$.}
  \label{fig:walkNP}
\end{figure}
On the polytope $\cP$, this means walking from vertex $P^{w-te_i}$ to
$P^{w-t' e_i}$ for scalars $t'>t > 0$ corresponding to points between
three consecutive intersection points, along an edge whose $i$-th
coordinate is negative. If the vector $w$ is generic, then we can
assume that $\sigma_j$ and $\sigma_k$ are parallel whenever they share
an intersection point obtained by shooting from $w$ in a fixed
coordinate direction. So we can use any of the cones in a parallel
class to compute the edge direction of the wall we walk across.  By
adding up the multiplicities of the cones in each class, we get the
lattice length of the edge of $\cP$.  This allows us to compute the
coordinates of the vertices dual to the chambers we walk into.  For
each of these vertices found by walking from a known vertex, we also
get an objective vector in the process.  For example, any vector of
the form $w - t e_i$, where $ t^i_{k_j} < t < t^i_{k_j + 1}$, is an
objective vector for the $j$-th vertex found in the walk in direction
$-e_i$. We take $t^i_{m+1}$ as $\infty$.  For numerical stability, we
use exact arithmetic over the rational numbers.  In particular, we
always choose the new objective vectors to be integral.

Using a new vertex, with its associated objective vector, we can repeat the ray-shooting (Algorithm \ref{alg:ray}) and
walking (Algorithm \ref{alg:walk}) again. The picture one should have in mind is that walking
from chamber to chamber in the tropical side corresponds to walks from
vertex to vertex in $\NP(f)$ along edges normal to the codimension one
cones traversed in the tropical hypersurface.

The combination of Algorithms 1 and 2 is illustrated in
Figures~\ref{fig:walkAndShoot} and~\ref{fig:walkNP}. Starting from
chamber $\mathcal{C}_0$ and an objective vector $w_0$, we shoot rays
in minus the coordinate axes directions. The intersection points are
indicated by their defining parameters $t^i_j$ (note that superscripts
are omitted in the notation of Algorithm\ref{alg:walk}). As we explain
below, to speed up the computation of Algorithm~\ref{alg:ray} we first
precompute the inverses of all suitable matrices of the form
$M_{\sigma}:=(-e_i| r_1 | \ldots |r_{15})$ where $\{r_1, \ldots,
r_{10}\}$ are generators of the cone $\sigma$ and $\{r_{11}, \ldots,
r_{15}\}$ span the lineality space in~\eqref{eqn:lineality}.  Using
this, the condition $(w_0-\lambda e_i)\cap \sigma \neq \emptyset$
translates to the first eleven coordinates of the solution $X$ of
$X^t=(M_{\sigma})^{-1}\cdot w_0$ being positive. Thus, we can easily
use the same systems to test $(w_0 + \lambda e_i) \cap \sigma \neq
\emptyset$, just changing the sign condition for the first coordinate
of $X$. This small modification allows us to walk in sixteen new
directions (the positive coordinate axes), and find new adjacent
vertices to vertex $v_0$ starting form objective vector $w_0$.  The
step updating $v$ in Algorithm~\ref{alg:walk} should be $ v \leftarrow
v \bm{+} \left(\sum_{k_{j-1} < k \leq k_j} m_{\sigma_k}\right) \cdot
\ell^{\sigma_{k_j}}$ instead of $ v \leftarrow v - (\ldots)$ .

In Figure~\ref{fig:walkAndShoot}, the parameters $\lambda$ associated
to the intersection points in these positive directions are denoted by
$\tilde{t}^i_j$. The dashed arrows indicate the shooting
directions. The points in the cones correspond to intersection points,
whereas the points inside chambers are the objective vectors obtained
for each vertice as described in Algorithm~\ref{alg:walk}.

The dual walk in the Newton polytope is depicted in
Figure~\ref{fig:walkNP}. We start walking from vertex $v_0$ and via
shooting we obtain the adjacent vertices $v_5, v_7, v_8$ and
$v_{10}$. Notice that by this procedure we miss vertices $v_4, v_6$
and $v_9$. However, we do get them if we start shooting from known
adjacent chambers to $\mathcal{C}_0$. For example, $v_6$ can be
computed if we shoot rays from chamber $\mathcal{C}_8$, followed by a
shoot from chamber $\mathcal{C}_6$. Observe that this depends heavily
on the choice of the objective vector $w_6$.

\subsection{Implementation}

A few notes about the implementation of our algorithms are in order.
As we started working on the problem, we used \texttt{Macaulay 2}
\cite{M2} to do the ray-shooting (Algorithm \ref{alg:ray}). This
script was fine for our first experiments, but it took three days to
generate a single vertex of the polytope. It soon became evident that
something faster was needed if we wanted to compute the entire
polytope.

Our first step was to translate the \texttt{Macaulay 2} script for
Algorithm~\ref{alg:ray} into Python~\cite{python}. We chose that
language because of its fast speed of development and availability of
arbitrary precision integers, which were needed by our program.  We
always scale our objects (matrices and vectors) by positive integers
so that our objects have integer coefficients. This step is crucial
for numerical stability.

This new implementation brought the running time to about 10
hours. This was a remarkable improvement, but as the number of
vertices of the polytope grew, we realized that something even faster
was required. Therefore, we decided to resort to caching: instead of
computing every inverse for each vector, we precomputed all the
inverses and stored them using a binary format suitable for fast
reading in Python (Pickles). This resulted in a file of a few tens of
gigabytes, but dropped the time required for an individual ray-shooting
procedure down to under three hours.

Once the Python prototype was working at a reasonable speed, we
translated it into C++~\cite{c++}, which brought the time required to do 
ray-shooting for a single vertex to 47 minutes on modest hardware. Moreover, ray-shooting for multiple objective vectors could be performed at the same time, thus amortizing the disk reads. Since
we still needed large integers, we decided to use GMP~\cite{gmp} and its C++ interface.

The procedure for walking is a more or less straightforward
translation of the pseudocode presented in
Algorithm~\ref{alg:walk}. It is still implemented in Python,
because it takes a short amount of time to walk from a few hundred
vertices at a time, and the simplicity of the script far outweights the
time gains a C++ translation would provide.

\subsection{Certifying facets}

We now discuss how to certify certain inequalities as facets of a polytope $\cP$ given by the dual tropical hypersurface $\cT(f)$.
By the duality between tropical hypersurfaces and Newton polytopes, each
facet direction must be a ray in the tropical variety, equipped with
the fan structure dual to $\cP $. Lemma~\ref{lm:rayTest} provides a
characterization for a vector in $\RR^n$ to be a ray of $\cT(f)$
with the inherited fan structure. 

\begin{lemma}\label{lm:rayTest}
  Let $w \in \RR^n$ and $\cT(f)$ be a tropical hypersurface given by a
  collection of cones, but with no prescribed fan structure. Let $d$
  be the dimension of its lineality space. Let
  $\mathcal{H}=\{\sigma_1, \ldots, \sigma_l\}$ be the list of cones
  containing $w$. Let $q_i$ be the normal vector to cone
  $\sigma_i$ for $i=1, \ldots, l$.
  Then, $w$ is a ray of $\cT(f)$ if and only if $\{q_1, \ldots, q_l\}$
  generates a $(n-d-1)$-dimensional vector space if and only if $w$ is a
  facet direction of $\NP(f)$.
\end{lemma}
\begin{proof}
The vectors $\{q_1, \dots, q_l\}$ are precisely the directions of edges in the face $\cP^w$ of $\cP := \NP(f)$.  Since the lineality space of $\cT(f)$ has dimension $d$, the polytope $\cP$ has dimension $n-d$.  The face $\cP^w$ is a facet of $\cP$ if and only if $q_1, \dots, q_l$ span a $(n-d-1)$-dimensional vector space.
\end{proof}

For any objective vector $w \in \RR^n$, we can compute a vertex in the
face $\cP^w$ by applying ray-shooting (Algorithm \ref{alg:ray}) to a
generic objective vector $w'$ in a chamber of the normal fan of $\cP$
containing $w$.  If we know that $w$ is in fact a facet direction of
$\cP$, then any vertex in $\cP^w$ gives us the constant term $a$ in
the facet inequality $w \cdot x \leq a$.  This is used in
Algorithm~\ref{alg:certifyFacets} for checking if a given inequality
is a facet inequality of $\cP$. This step will be essential to certify
that our partial list of vertices is indeed the complete list of
vertices of the polytope $\cP$. We discuss this approach in Section~\ref{sec:completing-polytope}.

\begin{algorithm}
\KwIn{An inequality $w \cdot x \leq a$ and a tropical hypersurface (dual to polytope $\cP$) given as a collection $\cF$ of maximal cones.
}
  \KwOut{\emph{True} if the inequality is a valid facet inequality of $\cP$;
    \emph{False} otherwise.} 
  \smallskip 
$\mathcal{N}\leftarrow \{\}$ \\
	\For{$\sigma \in \mathcal{F}$}{
		\If{$w \in \sigma$}{ 
$\mathcal{N}\leftarrow \mathcal{N}\cup \{$normal vector to $\sigma\}$;}}
\If{$ \dim \langle \mathcal{N}\rangle < n-d-1$}{{\bf Output} \emph{False}}
\Else{
$w' \leftarrow $ a vector in the interior of a chamber containing $w$\\
Compute the vertex $\cP^{w'}$ using ray-shooting (Algorithm \ref{alg:ray}).\\
\If{$w \cdot \cP^{w'}= a$} {{\bf Output} \emph{True}}
\Else{{\bf Output} \emph{False}}
} 
        \caption{Facet certificate: Check if a given inequality defines a facet of a polytope given by its normal fan.\label{alg:certifyFacets}}
\end{algorithm}

We now explain how to obtain a vector in the interior of a chamber
containing a facet direction $w$. We start by applying a modified version of
Algorithm~\ref{alg:ray} with input vector $w$ and when we choose to
shoot rays only in direction $-e_1$.  Since $w$ is a ray of the tropical
variety given by the collection $\mathcal{F}$, it belongs to some
cones $\{\tau_1, \ldots, \tau_s\}$ in $\mathcal{F}$. Let $\sigma_1,
\ldots, \sigma_m$ be the cones we intersect along the $-e_1$ direction
(we allow intersections at
boundary points of each cone). Note that we only pick those cones
with $l_{1}^{\sigma_j}\neq 0$. 

Now, we use Algorithm~\ref{alg:walk} with input vector $w$ and the set
$\mathcal{S}$ corresponding to the cones $\sigma_1, \ldots, \sigma_m$
and coordinate $1$. We assume $(\sigma_k, 1, t_k)$ are ordered in
increasing order, with all $t_k\geq 0$. We have two possible
scenarios: either $\mathcal{S}$ is a subset of $\{0\}$ (that is,
either the empty set or the set $\{0\}$) or it contains a positive
real number. In the first case, we pick an objective vector $w_1=w- t
e_1$ for a positive number $t$ (for numerical stability, we choose $t$
to be a big rational number). In the second case, pick a number $t$
between zero and the first positive number $t_j$ from $\mathcal{S}$
and let $w_1=w-te_1$.

Third, we check if any cone in $\mathcal{F}$ contains $w_1$ or not. If
not, then we let $w'=w_1$. If yes, by the balancing condition, this
means that there exists a maximal cone in the tropical variety
containing both $w_1$ and $w$. Note that this cone may be obtained by
gluing and/or subdividing some cones in $\mathcal{F}$. In this case
then we proceed as above, replacing the original input vector $w$ by
$w_1$ and shooting rays using coordinate $2$ instead of coordinate $1$. We
repeat this process with all coordinates if necessary. Unless we have
$w_i$ not contained in any cone of $\mathcal{F}$, at step $i$ we are
guaranted to have a cone containing $w, w_1, \ldots, w_{i-1}, w_i$ by
construction.  By dimensionality argument, at most in sixteen steps,
we obtain a vector $w_i$ not contained in any cone of
$\mathcal{F}$. This vector will be the objective vector $w'$ from
Algorithm~\ref{alg:certifyFacets}.

\subsection{Completing the polytope}
\label{sec:completing-polytope}
Once the ratio of new vertices computed with ray-shooting and walking
decreases, the next natural question that arises is how to guarantee
that we have found all vertices of our polytope. To answer this
question, we construct the tangent cones at each vertex and try to certify their facets as facets of $\cP$.

\begin{definition}
  Let $\cP $ be a full-dimensional polytope in $\RR^N$ and $v$ a vertex
  of $\cP $. We define the \emph{tangent cone of $\cP $ at $v$} to be
  the set:
\[
\cT^{\cP }_v:= v + \RR_{\geq 0}\langle w-v: w\in \cP \rangle = 
v + \RR_{\geq
  0} \langle e: e \text{ edge of } \cP  \text{ adjacent to }v\rangle. 
\]
\end{definition}

By construction, $\cT^{\cP }_v$ is a polyhedron with only one vertex
and $\cP\!\! =\!\! \bigcap_{v \text{ vertex of }\cP} \cT^{\cP }_v\!\!$.  In
particular, an inequality defines a facet of $\cP$ if and only if it
defines a facet of one of the tangent cones.

Let $\mathcal{Q}$ be the convex hull of the vertices of $\mathcal{P}$
obtained via Algorithms~\ref{alg:ray} and~\ref{alg:walk}. Our goal is
to certify that $\mathcal{Q} = \mathcal{P}$.  We proceed as
follows. For each vertex $v$ of $\cQ$ we wish to compare the tangent
cones $\cT_v^\cQ$ and $\cT_v^\cP $.  Since $\cQ$ has over seventeen
million vertices and $\cT^\cQ_v$ has no symmetry, straightforward
convex hull computations are infeasible.  If $\cT^\cQ_v = \cT^\cP _v$
then the extreme rays of $\cT^\cQ_v$ would be edge directions of $\cP
$, which we have already computed as the normal directions to the
maximal cones of the tropical hypersurface, and which are 
15\,788 in total. For a fixed vertex $v \in \cQ$ we
compute all differences $w-v$ for all vertices $w$ of $\cQ$ and test
which of these vectors are parallel to edges of $\cP $. The number of
such edge directions in $\cT^\cQ_v$ is expected to be very small
(usually under 30 in practice).  Let $C_v^{\cQ,\cP}$ be the convex
hull of $v$ and all rays along the edge directions of $\cP$ in
$\cT^{\cQ}_v$.  So we have $C_v^{\cQ,\cP} \subseteq \cT_v^{\cQ}$ and
we can test if $C_v^{\cQ,\cP} \supseteq \cT_v^{\cQ}$ by computing
facets of $C_v^{\cQ,\cP}$ with \texttt{Polymake}~\cite{polymake}.  If
$C_v^{\cQ,\cP} \supseteq \cT_v^{\cQ}$, we use Algorithm
\ref{alg:certifyFacets} to check whether each facet of $C_v^{\cQ,\cP}$
is also a facet of $\cP $.  In this way, we can certify that
$\cT_v^{\cP } \subseteq C_v^{\cQ,\cP}$, hence $C_v^{\cQ,\cP} =
\cT_v^{\cQ} = \cT_v^{\cP }$.  Certifying this for a vertex $v$ of
$\cQ$ in each symmetry class will give us $ \cQ = \bigcap_{v \textrm{
    vertex of }Q} \cT^\cQ_v \supseteq \bigcap_{v \textrm{ vertex of
  }P} \cT^\cP _v = \cP $, hence $\cQ = \cP $.  We conclude:
\begin{lemma}\label{lm:certificateCompletePolytope}
  Let $\cP $ be a polytope and $\cQ\subset \cP $ be the convex hull of
  a subset of the vertices in $\cP $. If all facets of $\cQ$ are
  facets of $\cP $, then $\cQ \supset P$, so $\cQ=\cP $.
\end{lemma}

If we find that a facet $w\cdot x \leq a$ of $C_v^{\cQ,\cP}$ is not a facet
of $\cP $ from Algorithm \ref{alg:certifyFacets}, then we are missing
vertices adjacent to $v$ in $\cP $ in this ``false facet direction''
$w$, so we can perturb $w$ so that it lies in a chamber of the normal
fan of $\cP $ and use ray-shooting (Algorithm \ref{alg:ray}) to find a
new vertex in that direction.  Using this method, we obtained the entire
polytope in finite number steps. We describe the process of
approximating $\cP $ by a subpolytope $\cQ$ in
Algorithm~\ref{alg:approximateP}.
A schematic of complete tangent cones and incomplete tangent cones is
depicted in Figure~\ref{fig:approximatePolytope}. 

% In the final stages of the computation, if we find that
% $C_v^{\cQ,\cP}$ is a strict subcone of the tangent cone $ \cT^\cQ_v$,
% we enumerated the rays $w-v$ (with $w \in V$) that lie in the
% difference $ \cT^\cQ_v \backslash C_v^{\cQ,\cP}$. If the number of
% such rays is small (no more than a few hundreds), we replace
% $C_v^{\cQ,\cP}$ with the convex hull of $C_v^{\cQ,\cP}$ and those rays
% (computed using {\polymake}) and proceed as in
% Algorithm~\ref{alg:approximateP}.  By executing
% Algorithm~\ref{alg:approximateP} in this way, we were able to compute
% and certify all vertices and facets of the polytope.

\begin{algorithm}
  \KwIn{A partial list $V$ of vertices of $\cP $, a collection of cones
    $\mathcal{F}$ whose union is the tropical hypersurface, 
    $d=$dimension of lineality space of the tropical hypersurface, and
    the group of symmetries of the tropical hypersurface.  }
  \KwOut{A complete list of vertices and facets of $\cP $.}
$\mathcal{S} \leftarrow \{\}$;\\
  \For{representatives $v$ of orbits of $V$}{
  	 $C^{\cQ, \cP}_v \leftarrow $ convex hull of $v$ and all rays in directions $w - v$ where $w \in V$ and $w - v$ is normal to a cone in $\cF$.\\
	   $A \leftarrow $ facets of $C^{\cQ, \cP}_v $ (using {\tt Polymake}).\\
    \For{$z \in A$}{
\If { $z$ is a not facet of $\cP $ by Algorithm~\ref{alg:certifyFacets}} {
$w' \leftarrow $ a vector in the interior of a chamber whose closure contains $z$.\\
Compute the vertex $\cP^{w'}$ using ray-shooting (Algorithm \ref{alg:ray}).\\
$V \leftarrow V\, \cup $ orbit of $\cP^{w'}$\\
{\bf Break} and restart the outermost for-loop with the new $V$.
}
\Else{$\mathcal{S}\leftarrow \mathcal{S} \cup  \{z\}$}}

  {\bf Output} vertices $V$ and facets $\mathcal{S}$.}

        \caption{Approximation of $\cP $ by a subpolytope $\cQ$: Given a
          partial list of vertices of a polytope $\cP $ with no known
          complete list of vertices, we construct the subpolytope
          $\cQ$ generated by this list. We certify
          when $\cQ$ equals $\cP $.\label{alg:approximateP}}
\end{algorithm}
\begin{figure}[htb]
  \centering
  \includegraphics[scale=0.4]{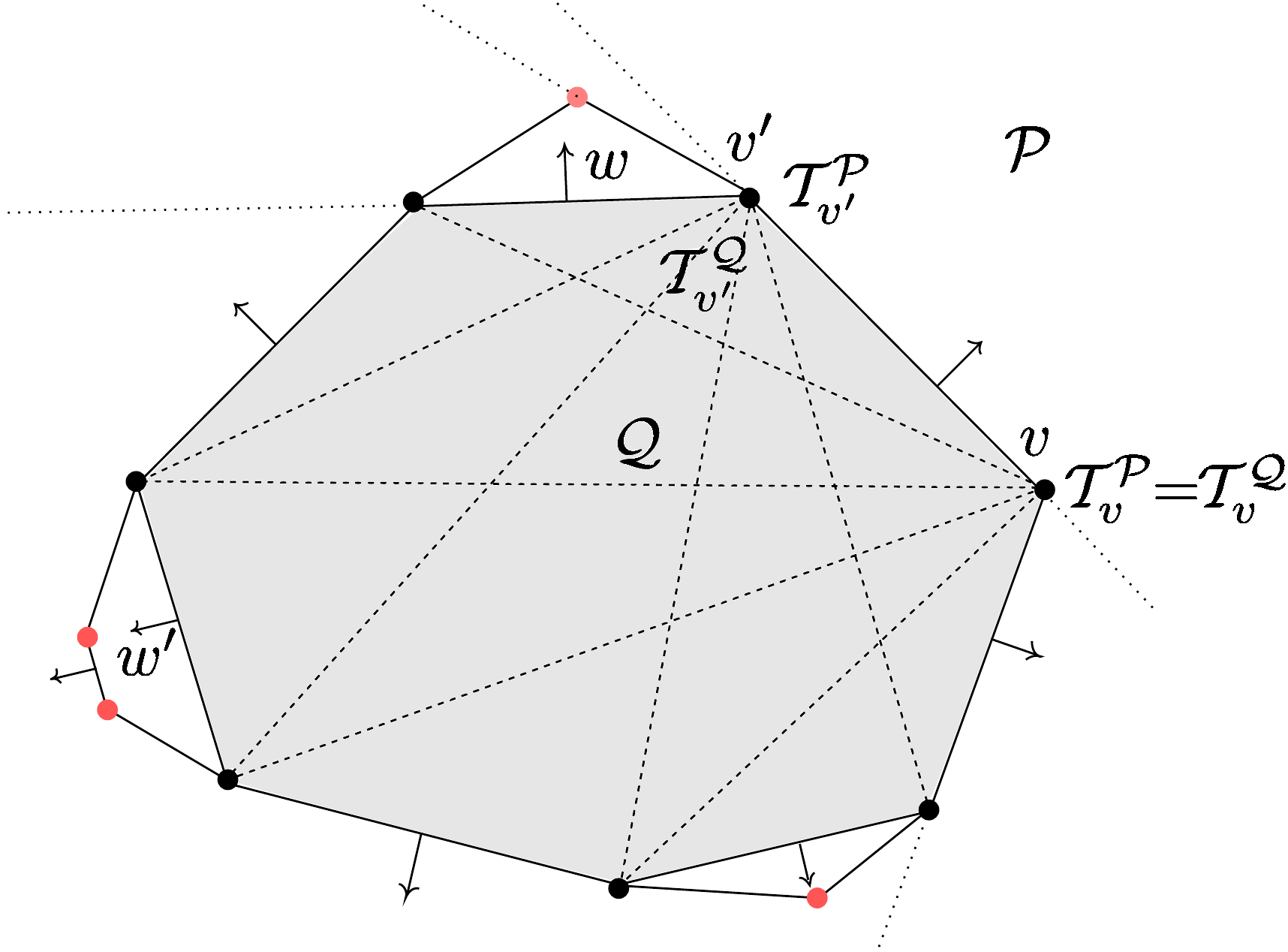}
 \caption{Approximation algorithm. We compute the tangent cones at
    vertices of $\mathcal{P}$ and we build a polytope
    $\mathcal{Q}\subset \mathcal{P}$. We certify if facets of
    $\mathcal{Q}$ are also facets of $\mathcal{P}$ by
    Algorithm~\ref{alg:certifyFacets}. In the picture, we certify all
    facet directions in $\cT^{\mathcal{Q}}_v$ containing vertex $v$
    but we cannot certify the facet direction $w$ of the tangent
    cone $\cT^{\mathcal{Q}}_{v'}$. In addition, although we can
    certify the facet direction $w'$ of $\mathcal{Q}$ as a true facet
    direction of $\mathcal{P}$, we will not be able to certify
    the constant corresponding to this facet direction of $\mathcal{P}$ since we are missing all
  its supporting vertices. The true constant will be obtained using
  Algorithm~\ref{alg:certifyFacets}.} 
  \label{fig:approximatePolytope}
\end{figure}
%\vspace{-2ex} 

In the final stages of the computation, if we find that
$C_v^{\cQ,\cP}$ is a strict subcone of the tangent cone $ \cT^\cQ_v$,
we enumerated the rays $w-v$ (with $w \in V$) that lie in the
difference $ \cT^\cQ_v \backslash C_v^{\cQ,\cP}$. If the number of
such rays is small (no more than a few hundreds), we replace
$C_v^{\cQ,\cP}$ with the convex hull of $C_v^{\cQ,\cP}$ and those rays
(computed using {\polymake}) and proceed as in
Algorithm~\ref{alg:approximateP}.  By executing
Algorithm~\ref{alg:approximateP} in this way, we were able to compute
and certify all vertices and facets of the polytope.

%\vspace{-2ex}
%%%%%%%%%%%%%%%%%%%%%%%%%%%%%%%%%%%%%%%%%%%%%%%%%%%%%%%%%%%%%%%%%%%%%%%%

\section*{Acknowledgment}

%%%%%%%%%%%%%%%%%%%%%%%%%%%%%%%%%%%%%%%%%%%%%%%%%%%%%%%%%%%%%%%%%%%%%%%%

We wish to acknowledge Bernd Sturmfels for suggesting this problem. We
thank Dustin Cartwright, Daniel Erman and Anders Jensen for
inspiring discussions and our two anonymous
referees for helping us improve the exposition.  We also thank the \emph{Mathematical Sciences Research Institute} (MSRI) for providing a wonderful working environment for this project, and the computing staff at MSRI and Georgia Tech School of Math for their fantastic support.  Finally we acknowledge
 the computers at MSRI, Georgia Tech, and the University of Buenos Aires for their hard work.

%%%%%%%%%%%%%%%%%%%%%%%%%%%%%%%%%%%%%%%%%%%%%%%%%%%%%%%%%%%%%%%%%%%%%%%%
%%%%%%%%%%%%%%%%%%%%%%%%%%%%%%%%%%%%%%%%%%%%%%%%%%%%%%%%%%%%%%%%%%%%%%%%

\bibliographystyle{plain}

\end{document}